\documentclass[a4paper,12pt]{article}
\usepackage[latin1]{inputenc}
\usepackage[T1]{fontenc}

\usepackage{amsmath}
\usepackage{amsthm}
\usepackage{amssymb}
\usepackage{bm}

\usepackage{graphicx}
\usepackage{multirow}
\usepackage{natbib}
\usepackage[english]{babel}
\usepackage[margin=2cm]{geometry}
\DeclareMathOperator*{\diag}{diag}

\newcommand{\Cov}{\mathrm{Cov}}
\newcommand{\cov}{\mathrm{cov}}

\newtheorem{lemma}{Lemma}
\newtheorem{theorem}{Theorem}
\newtheorem{corollary}{Corollary}

\usepackage{xcolor}

\usepackage{soul}
\usepackage{xifthen}

\allowdisplaybreaks

\begin{document}
\begin{center}
  \Large \textbf{Optimal Design for Estimating the Mean Ability over Time
    in Repeated Item Response Testing}
\end{center}
\begin{center}
  \textbf{Fritjof Freise\footnote{corresponding author}\\
    University of Veterinary
    Medicine Hannover, Department of Biometry, Epidemiology and
    Information Processing,\\
    B\"unteweg 2,
    30559 Hannover, Germany,\\
    e-mail: fritjof.freise@tiho-hannover.de}
\end{center}
\begin{center}
  \textbf{Heinz Holling\\
    University of M\"unster, Institute for Psychology,\\
    Fliednerstra{\ss}e 21, 48149 M\"unster, Germany,\\
    e-mail: holling@wwu.de}
\end{center}
\begin{center}
  \textbf{Rainer Schwabe\\
    University of Magdeburg, Institute for Mathematical Stochastics,\\
    Universit\"atsplatz 2, 39106 Magdeburg, Germany,\\
    e-mail: rainer.schwabe@ovgu.de}
\end{center}
\begin{abstract}
  We present general results on D-optimal designs for estimating the
  mean response in repeated measures growth curve models with metric
  outcomes. For this situation, we derive a novel equivalence theorem
  for checking design optimality. The motivation of this work
  originates from designing a study in psychological item response
  testing with multiple retests to measure the improvement in
  ability. Besides introductory linear growth curves for which
  analytical results can be obtained, we consider two non-linear
  growth curve models incorporating an increasing mean ability and a
  saturation effect. For these models, D-optimal designs are
  determined by computational methods and are validated by means of
  the equivalence theorem.

  Keyword:
  D-optimal design; growth curve; repeated measures model; mixed effect model

  MSC classification:
  Primary: 62K05
  Secondary: 62J05, 62J02, 62P15
\end{abstract}

\section{Introduction}
This article is motivated by planning experiments for analyzing retest
effects of cognitive abilities.  These effects are an important issue
in educational sciences since a long time.  They refer to score gains
through repeated applications of cognitive ability tests, e.g., the
Scholastic Aptitude Test.  Due to retesting a considerable increase in
test performance may arise which has necessarily to be considered to
adequately interpret the test results. Very often, retesting of
intelligence tests has been studied.  According to the results of a
recent meta-analysis by \cite{ScharfenPetersHolling:2018} retesting of
general intelligence leads to significant score gains proceeding
nonlinearly over the first three test administrations and then
reaching a plateau.  Furthermore, a relatively large standard
deviation of the true effects was observed, meaning that the mean
increase of the score gains considerably differs between subjects.
Such score gains are usually analyzed by latent growth models
representing a class of statistical methods to estimate growth or
changes over a period of time.  These statistical techniques are
widely used in psychology and in other behavioral as well as social
sciences.  Usually, the longitudinal data to be analyzed are based on
the same subjects measured at different time points by the same tests
or by parallel test versions.  The growth curves, also called latent
trajectories, might take on different forms, e.g., they may be
systematically increasing or decreasing over time.  Often, a mean
trajectory will be estimated as well as individual deviations from
this trajectory.

The purpose of the present work is to investigate optimal designs for
estimating the model parameters for the mean score.  We first give a
general model description for the response in Section~\ref{sec:model}.
There a key component of the model is that all subjects are observed
repeatedly.  This leads to models with correlated observations within
subjects.  For the within subject covariance we will consider two
different covariance structures: compound symmetry and autoregressive.
In the model assumption a further important feature is that only the
number of time points taken is important and that the actual time
elapsed between different testings does not play a role.  In
Section~\ref{sec:linear-mixed-effects-new}, we derive the covariance
matrices for the maximum likelihood estimates of the mean score.
First, we consider an unstructured behavior of the mean score
described by an analysis of variance type model which treats the time
points as levels of a single factor.  Then, we introduce a straight
line model for the mean response curve and two nonlinear models which
capture the saturation property of a plateau: a sigmoid (logistic
type) and an exponential decay model.  These latter models are common
candidates in modeling dose response relationships (see
e.g. \citealp{BrensonPinheiroBretz:2003TRa} or
\citealp{DetteBretzPepelyshevPinheiro:2008a}).  Here, we also derive a
general form of the information matrix which relies only on the
frequencies of observations at each time point.  These frequencies
completely determine the properties of a design in the present
situation.  The quality of a design is measured in terms of the Fisher
information matrix since the (asymptotic) covariance matrix of the
estimates is proportional to the inverse Fisher information.  In
Section~\ref{sec:optimal-design} we summarize some basic facts of the
theory of optimal design and obtain analytical results in the case of
unstructured and linear mean response curves.  We also establish there
an equivalence theorem which provides a tool for checking the
optimality of a candidate design.  For nonlinear models of the mean
response curve, the information matrix may depend on the model
parameters \citep[see e.g.][Chapter 6]{Silvey:1980}.  In that case, we
restrict ourselves to locally optimal designs which may depend on
nominal values for the parameters \citep{Chernoff:1953}.  In
Section~\ref{sec:computational} we present numerical results of
optimal designs for nonlinear mean response curves when analytical
solutions are not available.  The article concludes with a short
discussion in Section~\ref{sec:discussion}.  Technical proofs are
deferred to the Appendix.

\section{General Model Specifications}
\label{sec:model}
As outlined in the introduction, the model has to describe the
situation where \(N\) subjects are tested at \(J\) time points
\(t_1 < \ldots < t_J\) each. The test administered at the $j$th time
point \(t_j\) consists of the same number $I_{j}$ of items for each
subject. The total number of items {per subject} over all time points
will be denoted by $I = \sum_{j = 1}^J I_{j}$.  The corresponding
observation of the score \(y_{n j i}\) of subject \(n\) for the
\(i\)th item at time point \(t_j\) is modeled as the realization of a
random variable \(Y_{n j i}\).  For each subject the mean of
\(Y_{n j i}\) is given by the individual ability \(\mu_{nj}\) of
subject \(n\) at time point \(t_j\).

Moreover, as we are interested in a general description of the
development of ability over time, we assume that at each time point
\(t_j\) the individual abilities \(\mu_{n j}\) deviate randomly from
an aggregate population ability \(\mu_{j}\) at time point \(t_j\).

As a general framework, we assume, that for all time points \(t_j\)
the population ability \(\mu_{j}\) is a function of a
\(p\)-dimensional vector \(\bm{\beta} =
\begin{pmatrix}
  \beta_{0} & \cdots & \beta_{p - 1}
\end{pmatrix}^{\mathrm{T}}\) of parameters, i.e.\
\(\mu_{j} = \mu_{j}(\bm{\beta})\) and \(p \leq J\).  The corresponding
vector containing the population abilities for all time points is
denoted by \(\bm{\mu}(\bm{\beta}) =
\begin{pmatrix}
  \mu_{1}(\bm{\beta}) & \cdots & \mu_{J}(\bm{\beta})
\end{pmatrix}^{\mathrm{T}}\).  Specifically, in
Section~\ref{sec:linear-mixed-effects-new}, we will start with
considering an unstructured model in which the population ability
\(\bm{\mu} = \bm{\beta}\) may vary arbitrarily over time without any
restrictions, i.e.\ \(p = J\), and then extend the results to
structured models in Subsection~\ref{subsec:longmodeltheta}, where the
population ability depends on time by a functional, potentially
nonlinear relationship (\(p < J\)).

While observations can be assumed to be uncorrelated between subjects,
they will, typically, be correlated within subjects. For modelling
this dependence, we introduce the individual abilities \(\mu_{nj}\) as
random effects \(\mu_{n j} = \mu_{j} + \gamma_{n j}\), where
\(\gamma_{n j}\) are the individual deviations from the population
mean \(\mu_{j}\) for subject \(n\) at time point \(t_j\).  To become
more specific, we suppose that the conditional mean of the score
\(Y_{n j i}\) given the random effect can be described as the
difference of the ability \(\mu_{n j}\) of subject \(n\) at time point
\(t_j\) and the difficulty \(\varsigma_{j i}\) of the item presented,
i.e.\
\(\mathrm{E}(Y_{n j i}|\gamma_{n j}) = \mu_{n j} - \varsigma_{j i} =
\mu_{j} + \gamma_{n j} - \varsigma_{j i}\).

The difficulties \(\varsigma_{j i}\) of the items presented to the
subjects are supposed to be known from previous calibration
experiments. As the item difficulties are modeled as additive
constants, we may transform the scores to their normalized version
\(\widetilde{Y}_{n j i} = Y_{n j i} - \varsigma_{j i}\). We, thus, may
assume without loss of generality that throughout the remainder of the
paper all difficulties are formally set to \(0\) and, hence,
\(\mathrm{E}(Y_{n j i}|\gamma_{n j}) = \mu_{n j} = \mu_{j} + \gamma_{n
  j}\).  As a consequence, only the numbers \(I_j\) of items at each
time point \(t_j\) are relevant for design considerations.

In the present situation, the individual observations are modeled by
the linear mixed model
\begin{equation*}
  \label{eq:lmm2}
  Y_{n j i}  
  = \mu_{j} + \gamma_{n j} +
  \varepsilon_{n j i},
\end{equation*}
$i = 1, \ldots, I_{j}$, $j = 1, \ldots, J$, $n = 1, \ldots, N$. The
errors $\varepsilon_{n j i}$ are assumed to be uncorrelated and
normally distributed with zero mean and variance
\(\sigma_{\varepsilon}^2 > 0\).

Also the random effects \(\gamma_{nj}\) are assumed to be jointly
normally distributed with zero mean, and to be independent between
subjects, i.e.\ \(\cov(\gamma_{n j}, \gamma_{n' j'}) = 0\) for
\(n \neq n'\), \(j, j' = 1, \ldots, J\), but that there is a within
subject covariance \(\sigma_{j j'}\), i.e.\
\(\cov(\gamma_{nj}, \gamma_{n j'}) = \sigma_{j j'}\), \(j\neq j'\),
\(n = 1, \ldots, N\). Also assume, that the errors and random effects
are uncorrelated, i.e.\
\(\cov(\varepsilon_{n j i}, \gamma_{n' j'}) = 0\), for all
\(i = 1, \ldots, I_{j}\), \(j, j' = 1, \ldots, J\), and
\(n, n' = 1, \ldots, N\). Denote the vector of individual deviations
for subject \(n\) by
\(\bm{\gamma}_{n} =\begin{pmatrix} \gamma_{n 1} & \cdots & \gamma_{n
    J} \end{pmatrix}^{\mathrm{T}}\).  For these random effects, we
assume that they are independent identically multivariate normal with
zero mean and nonnegative definite covariance matrix
\(\mathbf{\Sigma}_{\gamma}\), with entries \(\sigma_{j j'}\).

For the within subject covariance structure we deal with two different
models which are commonly used to describe the within subject
dependence: Compound symmetry and an autoregressive (AR(1))
correlation structure.

In the case of compound symmetry, it is assumed that the correlation
between all repetitions is the same for all pairs of distinct time
points. Then the covariance matrix is given by
$\mathbf{\Sigma}_{\gamma,CS} = \sigma_{\gamma}^{2} ((1 - \rho)
\mathbf{I}_{J} + \rho \mathbf{1}_{J} \mathbf{1}_{J}^{\mathrm{T}})$,
where \(\sigma_\gamma^2 > 0\) is the variance of the random effects
$\gamma_{n j}$, \(\rho\) is the correlation of the random effects
within subjects, \(\mathbf{I}_{m}\) is the \(m \times m\) identity
matrix, and \(\mathbf{1}_{m} =
\begin{pmatrix}
  1 & \cdots & 1
\end{pmatrix}^{\mathrm{T}}\) denotes the \(m\)-dimensional vector with
all entries equal to \(1\).  We additionally assume that the within
subject correlation is nonnegative, \(0 \leq \rho \leq 1\).  Note that
compound symmetry occurs when the random effects split up additively
(\(\gamma_{n j} = \tilde\gamma_{n 0} + \tilde\gamma_{n j}\)) into a
common block effect \(\tilde\gamma_{n 0}\) with variance
\(\rho \sigma_{\gamma}^{2}\) which is constant over time and
independent identically distributed time effects
\(\tilde\gamma_{n j}\) with variance
\((1 - \rho) \sigma_{\gamma}^{2}\).

In the autoregressive covariance structure the correlation between
neighboring time points is higher than the correlation of those, which
are farther away.  More specifically, for AR(1), we have 
\(\sigma_{j j^\prime} =\sigma_\gamma^2 \rho^{|j - j^\prime|}\), i.e.\ 
\begin{equation*}
  \mathbf{\Sigma}_{\gamma,AR}=\sigma_{\gamma}^{2}
  \begin{pmatrix}
    1 & \rho & \rho^2 & \cdots & \rho^{J - 1} \\
    \rho & 1 & \rho & \cdots & \rho^{J - 2} \\
    \vdots & \ddots & \ddots & \ddots & \vdots \\
    \rho^{J - 2} & \cdots & \rho & 1 & \rho \\
    \rho^{J - 1} & \cdots & \rho^2 & \rho & 1
  \end{pmatrix}
\end{equation*}
with \(\sigma_{\gamma}^2>0\).  Also here we additionally assume that
the within subject correlation is nonnegative, \(0 \leq \rho \leq 1\).

To display distributional properties of the model and to derive
expressions for parameter estimates, it is convenient to rewrite the
model in vector notation.  The vectors of observations and errors for
subject \(n\) at time point \(t_j\) are
\(\mathbf{Y}_{n j} = \begin{pmatrix} {Y}_{n j 1}& \cdots & {Y}_{n j
    I_{j}} \end{pmatrix}^{\mathrm{T}}\) and
\(\bm{\varepsilon}_{n j} = \begin{pmatrix} \varepsilon_{n j 1} &
  \cdots & \varepsilon_{n j I_{j}} \end{pmatrix}^{\mathrm{T}}\),
respectively.

The vector \(\mathbf{Y}_{n j}\) of observations for subject \(n\) at
time point \(t_j\) is modeled by
\begin{equation*}
  \mathbf{Y}_{n j}
  =
  \mathbf{1}_{I_{j}}\mu_{j}
  + \mathbf{1}_{I_{j}}\gamma_{nj}
  + \bm{\varepsilon}_{n j},  
\end{equation*}
\(j = 1, \ldots, J\), \(n = 1, \ldots, N\).  Then \(\mathbf{Y}_{n j}\)
is \(I_j\)-dimensional multivariate normal with expectation
\(\mathrm{E}(\mathbf{Y}_{n j}) = \mathbf{1}_{I_{j}}\mu_{j}\) and
covariance matrix
\(\Cov(\mathbf{Y}_{n j}) = \sigma_{\varepsilon}^{2}\mathbf{I}_{I_{j}}
+ \sigma_{j j} \mathbf{1}_{I_{j}} \mathbf{1}_{I_{j}}^{\mathrm{T}}\).
Moreover, within subjects observational vectors have covariance
\(\Cov(\mathbf{Y}_{n j}, \mathbf{Y}_{n j'}) = \sigma_{j j'}
\mathbf{1}_{I_{j}} \mathbf{1}_{I_{j'}}^{\mathrm{T}}\),
\(j, j' = 1, \ldots, J\), \(j \neq j'\), \(n = 1, \ldots, N\), and
observational vectors are uncorrelated between subjects,
\(\Cov(\mathbf{Y}_{n j}, \mathbf{Y}_{n' j'}) = \mathbf{0}\) for
\(n \neq n'\), where \(\mathbf{0}\) is a generic zero matrix of
appropriate size.

Denote by
\(\mathbf{F}_{n} =\diag_{j = 1, \ldots, J}(\mathbf{1}_{I_{j}})\) the
individual design matrixfor subject \(n\), where
\(\diag_{j = 1, \ldots, J}(\mathbf{C}_{j})\) denotes a
\(\sum_{j = 1}^{J} m_{j} \times \sum_{j = 1}^{J} n_{j}\) block
diagonal matrix with diagonal blocks \(\mathbf{C}_{j}\) of size
\(m_{j} \times n_{j}\) and off-diagonal zero matrices \(\mathbf{0}\)
of appropriate size. Remember that all individuals receive the same
design, i.e.\ \(\mathbf{F}_{n}= \mathbf{F}\), say, for all
\(n = 1, \ldots, N\).  Then the model for all observations
\(\mathbf{Y}_{n} =
\begin{pmatrix} \mathbf{Y}_{n 1}^{\mathrm{T}} & \cdots & \mathbf{Y}_{n
    J}^{\mathrm{T}} \end{pmatrix}^{\mathrm{T}}\) of a subject \(n\)
can be written in vector notation as
\begin{equation*}
  \mathbf{Y}_{n}
  = \mathbf{F}
  \bm{\mu}
  + \mathbf{F}
  \bm{\gamma}_{n}
  + \bm{\varepsilon}_{n},
\end{equation*}
where \(\bm{\varepsilon}_{n} =
\begin{pmatrix} \bm{\varepsilon}_{n 1}^{\mathrm{T}}& \cdots &
  \bm{\varepsilon}_{n J}^{\mathrm{T}} \end{pmatrix}^{\mathrm{T}}\) is
the corresponding error vector for subject \(n\).  Then for each
subject the observation vector \(\mathbf{Y}_{n}\) is \(I\)-dimensional
multivariate normal with expectation
\begin{equation*}\mathrm{E}(\mathbf{Y}_{n}) = \mathbf{F}
  \bm{\mu}
\end{equation*}
and covariance matrix
\begin{equation*}\mathbf{V}_{n} = \Cov(\mathbf{Y}_{n})
  = 
  \sigma_{\varepsilon}^{2} \mathbf{I}_{I} 
  + \mathbf{F}
  \mathbf{\Sigma}_{\gamma}
  \mathbf{F}
  ^{\mathrm{T}} .
\end{equation*}
Note that also the individual covariances are identical,
\(\mathbf{V}_{n}=\mathbf{V}\), for all individuals
\(n = 1, \ldots, N\).  Further, because at each time point \(t_j\) all
subjects have to respond to the same number \(I_j\) of items, the
model equation of the vector
\(\mathbf{Y} = \begin{pmatrix} \mathbf{Y}_{1}^{\mathrm{T}} & \cdots &
  \mathbf{Y}_{N}^{\mathrm{T}} \end{pmatrix}^{\mathrm{T}}\) of all
observations has a product-type structure
\begin{equation*}
  \mathbf{Y}
  = (\mathbf{1}_{N} \otimes \mathbf{F}) \bm{\mu}
  + (\mathbf{I}_{N} \otimes \mathbf{F}) \bm{\gamma}
  + \bm{\varepsilon} \, .
\end{equation*}
where \(\bm{\gamma} = \begin{pmatrix} \bm{\gamma}_{1}^{\mathrm{T}} & \cdots &
  \bm{\gamma}_{N}^{\mathrm{T}} \end{pmatrix}^{\mathrm{T}}\)
is the \(NJ\)-dimensional stacked vector of all random effects, 
\(\bm{\varepsilon}  = \begin{pmatrix} \bm{\varepsilon}_{1}^{\mathrm{T}} & \cdots 
  & \bm{\varepsilon}_{N}^{\mathrm{T}} \end{pmatrix}^{\mathrm{T}}\)
is the \(NI\)-dimensional stacked vector of error terms, and
``\(\otimes\)'' denotes the Kronecker product of matrices and vectors, respectively.

Observations from different subjects \(n\) and \(n'\) are
uncorrelated, \(\Cov(\mathbf{Y}_{n}, \mathbf{Y}_{n'}) = \mathbf{0}\).
The vector of the random components and the vector of the error terms
are multivariate normal with zero mean and covariance matrix
\(\Cov(\bm{\gamma}) = \mathbf{I}_{N} \otimes
\mathbf{\Sigma}_{\gamma}\) and
\(\Cov(\bm{\varepsilon}) = \sigma_{\varepsilon}^2 \mathbf{I}_{N I}\),
respectively.  As a consequence, \(\mathbf{Y}\) is multivariate normal
with expectation
\(\mathrm{E}(\mathbf{Y}) = (\mathbf{1}_N\otimes \mathbf{F}) \bm{\mu}\)
and covariance matrix
\(\Cov(\mathbf{Y}) = \mathbf{I}_{N} \otimes \mathbf{V}\).  Note, that
the individual covariance matrix \(\mathbf{V}\) and, hence,
\(\Cov(\mathbf{Y})\) is non-singular, since
\(\mathbf{V} \geq \sigma_{\varepsilon}^{2} \mathbf{I}_{I} > 0\),
where, for matrices, the relations ``\(\geq\)'' and ``\(>\)'' are
meant in the sense of nonnegative and positive definiteness,
respectively.

\section{Mean Response Curves}
\label{sec:linear-mixed-effects-new}
To derive results for the situation that the mean responses
\(\mu_1, \ldots, \mu_J\) are emerging from a growth curve with only
few parameters, we first start with an unstructured case in which the
time points \(t_1,\ldots,t_J\) may be considered as levels in a
one-way layout model.

\subsection{Analysis of Variance Type Model (Unstructured  Time Dependence)}
\label{sec:sub:unstructured}
We start with assuming \(\bm{\mu} = \bm{\beta}\) without imposing any
restrictions on \(\mu_{j} = \beta_{j}\), \(j = 1, \ldots, J\).  Hence,
the time points may be interpreted as levels of a single factorial
variable ``time''.  Estimation of the mean score in this model results
in estimating \(\bm{\beta} = \bm{\mu}\).  Under the normality
assumption, the maximum-likelihood estimator is a generalized least
squares estimator which is given by
\begin{align}
  \label{eq:OLS}		
  \hat{\bm{\beta}} 
  & =
    \left((\mathbf{1}_{N} \otimes \mathbf{F})^{\mathrm{T}} (\mathbf{I}_{N} \otimes \mathbf{V})^{-1}
    (\mathbf{1}_{N} \otimes \mathbf{F})\right)^{-1}
    (\mathbf{1}_{N} \otimes \mathbf{F})^{\mathrm{T}} (\mathbf{I}_{N} \otimes \mathbf{V})^{-1}
    \mathbf{Y} 
    \nonumber
  \\
  & =
    {{\frac{1}{N}}}\sum_{n = 1}^{N}
    \left(\mathbf{F}^{\mathrm{T}} \mathbf{V}^{-1}
    \mathbf{F}\right)^{-1}
    \mathbf{F}^{\mathrm{T}} \mathbf{V}^{-1}
    \mathbf{Y}_n 
    =
    {{\frac{1}{N}}}\sum_{n = 1}^{N}
    \hat{\bm{\beta}}_n ,
\end{align}
where
\(\hat{\bm{\beta}}_n = \left(\mathbf{F}^{\mathrm{T}} \mathbf{V}^{-1}
  \mathbf{F}\right)^{-1} \mathbf{F}^{\mathrm{T}} \mathbf{V}^{-1}
\bm{Y_n}\) is the estimated mean score based on a single subject \(n\),
\(n = 1, \ldots, N\) (see e.g.\ \citealp{Rao:1973}, Chapter~4, for the
general structure, and \citealp{EntholznerBendaSchmelterSchwabe:2005},
for the representation as an average of individual fits).  As the
individual covariance matrix \(\mathbf{V}\) and the individual design
matrix \(\mathbf{F}\) interchange in the sense
\(\mathbf{V}\mathbf{F} = \mathbf{F}\mathbf{U}\), where
\(\mathbf{U} = \mathbf{\Sigma}_{\gamma} \diag_{j = 1, \ldots,
  J}(I_{j}) + \sigma_{\varepsilon}^2 \mathbf{I}_{J}\), the individual
generalized least squares estimator \(\hat{\bm{\beta}}_n\) and the
ordinary least squares estimator
\(\hat{\bm{\beta}}_{n,OLS} = \left(\mathbf{F}^{\mathrm{T}}
  \mathbf{F}\right)^{-1} \mathbf{F}^{\mathrm{T}} \bm{Y_n}\) based on
the observations of subject \(n\) coincide \cite[see][]{Zyskind:1967}.
Accordingly, also in the full model, the covariance matrix
\(\mathbf{I}_{N} \otimes \mathbf{V}\) and the design matrix
\(\mathbf{1}_{N} \otimes \mathbf{F}\) interchange,
\((\mathbf{I}_{N} \otimes \mathbf{V}) (\mathbf{1}_{N} \otimes
\mathbf{F}) = (\mathbf{1}_{N} \otimes \mathbf{F}) \mathbf{U}\). Thus,
also in the full model, the generalized least squares estimator
\(\hat{\bm{\beta}}\) and the ordinary least squares estimator
\(\hat{\bm{\beta}}_{OLS} =
\begin{pmatrix}
\hat\beta_1 & \cdots & \hat\beta_J
\end{pmatrix}^{\mathrm{T}}\) coincide, 
where \(\hat\beta_j = \frac{1}{N I_j}
\sum_{n = 1}^{N} \sum_{i = 1}^{I_j} Y_{n j i}\) is the average score at the \(j\)th time point over all subjects \(n = 1, \ldots, N\).
Thus, \(\hat{\bm{\beta}} = \hat{\bm{\beta}}_{OLS}\)
does not depend on the particular form of
\(\mathbf{V}\) and, hence, does not depend on \(\mathbf{\Sigma}_{\gamma}\).

The covariance matrix of
\(\hat{\bm{\beta}}\) is
\begin{equation}
  \label{eq:covoftheta}
  \mathbf{\Cov}(\hat{\bm{\beta}})
  = {\textstyle{\frac{1}{N}}} \left(\mathbf{F}^{\mathrm{T}} \mathbf{V}^{-1} \mathbf{F} \right)^{-1} .
\end{equation}
For the estimability of \(\bm{\beta}\) observations have to be made at
all time points, i.e.\ \(I_j > 0\) for all \(j = 1, \ldots, J\).  The
covariance matrix of \(\hat{\bm{\beta}}\) can be expressed as
\begin{equation}
  \mathbf{\Cov}(\hat{\bm{\beta}})
  = {\textstyle{\frac{1}{N}}} \left(
    \sigma_{\varepsilon}^{2}
    \mathbf{M}_{0}^{-1}
    + \mathbf{\Sigma}_{\gamma} 
  \right) , 
  \label{eq:covoftheta2}
\end{equation}
where
\(\mathbf{M}_{0} = \mathbf{F}^{\mathrm{T}} \mathbf{F} =
\mathrm{diag}_{j=1, \ldots, J}(I_j)\) is the information matrix in the
corresponding linear fixed effects model of a one-way layout with time
points \(t_j\) as levels.  The representation~\eqref{eq:covoftheta2}
is in accordance with the results given in
\cite{EntholznerBendaSchmelterSchwabe:2005}.

Note that the covariance matrix becomes smaller, in the sense of
nonnegative definiteness, if the numbers \(I_j\) of items are
increased or the number \(N\) of subjects becomes larger.  However,
consistency of the estimator is only achieved if the number \(N\) of
subjects tends to infinity while increasing the numbers \(I_j\) of
items will only reduce the covariance matrix to
\(\frac{1}{N} \mathbf{\Sigma}_{\gamma}\).

In the case of compound symmetry of the within covariance matrix
\(\mathbf{\Sigma}_{\gamma, CS}\) the covariance matrix of
\(\hat{\bm{\beta}}\) can be written as
\begin{equation*}
  \Cov(\hat{\bm{\beta}})
  = {{\frac{\sigma_{\gamma}^{2}}{N}}} \left(
    {\textstyle{\diag_{j = 1, \ldots, J}}}\left(1 - \rho + {\textstyle{\frac{1}{\tau^{2}I_{j}}}}\right)
    + \rho \mathbf{1}_{J} \mathbf{1}_{J}^{\mathrm{T}}
  \right),
\end{equation*}
where \(\tau^2 = \sigma_{\gamma}^2 / \sigma_{\varepsilon}^2\) denotes
the variance ratio.  If the number of items is equal at each time
point, i.e.\ \(I_j = I/J\) for all \(j = 1, \ldots, J\), then also
the covariance matrix of \(\hat{\bm{\beta}}\) can be seen to be
compound symmetric.

In the case of an autoregressive within covariance matrix
\(\mathbf{\Sigma}_{\gamma, AR}\) the covariance matrix of the
estimator \(\hat{\bm{\beta}}\) is given by
\begin{align*}
  \Cov(\hat{\bm{\beta}})
  & = \frac{\sigma_{\gamma}^{2}}{N}
    \begin{pmatrix}
      1 + \frac{1}{\tau^{2}I_{1}} & \rho & \rho^2 & \cdots & \rho^{J - 1} \\
      \rho & 1 + \frac{1}{\tau^{2}I_{2}} & \rho & \cdots & \rho^{J - 2} \\
      \vdots & \ddots & \ddots & \ddots & \vdots \\
      \rho^{J - 2} & \cdots & \rho & 1 + \frac{1}{\tau^{2}I_{J - 1}} & \rho \\
      \rho^{J - 1} & \cdots & \rho^2 & \rho & 1 + \frac{1}{\tau^{2}I_{J}}
    \end{pmatrix}.
\end{align*}

\subsection{Parameterized Growth Curves}
\label{subsec:longmodeltheta}
In this subsection we treat some commonly used parametric models for
the mean scores \(\mu_j(\bm{\beta})\).  For the time points \(t_j\) we
assume that they are equally spaced, \(t_{j} = j - 1\), because we are
not interested in the actual time between tests, but only in the
number of tests taken so far.

For illustrative purposes, we start with a simple straight line growth
curve.  The time trend is described by the equation
\(\mu_j(\bm{\beta}) = \beta_0 + \beta_1 t_j\), where \(\beta_0\) is
the mean score at the beginning of the testing period and \(\beta_1\)
is the increase in the mean score from one testing instance to the
next one.  This increment is assumed to be constant over time.

To cover the needs for a learning curve ending up in a plateau, we
further consider two nonlinear models for the mean response curve.
First, we consider a three parameter exponential decay model given by
\begin{equation}
  \mu_{j}(\bm{\beta}) = \beta_{1} - \left(\beta_{1} - \beta_{0}\right)
  \exp( - \beta_{2} t_{j})
  ,\quad j = 1, \ldots, J,
\end{equation}
and \(\bm{\beta} =
\begin{pmatrix}
  \beta_{0} & \beta_{1} & \beta_{2}
\end{pmatrix}^{\mathrm{T}}\), \(\beta_{0} < \beta_{1}\),
\(\beta_{2} > 0\).  In the exponential model the parameter
\(\beta_{0}\) corresponds to the initial value (offset) of the mean
response at the beginning of the testing period, \(t_1 = 0\).  The
asymptote (saturation level) for \(t \rightarrow \infty\) is given by
\(\beta_{1}\) and the strength of the slope (speed of saturation) is
determined by \(\beta_{2}\).

Alternatively, we also consider a sigmoid (four parameter logistic)
model given by
\begin{equation}
  \mu_{j}(\bm{\beta}) = \beta_{0} + \left(\beta_{1} - \beta_{0}\right)
  \frac{1}{1 + \exp( - (\beta_{2} t_{j} + \beta_{3}))}
  , \quad j = 1, \ldots, J,
\end{equation}
and
\(\bm{\beta} =
\begin{pmatrix}
  \beta_{0} & \beta_{1} & \beta_{2} & \beta_{3}
\end{pmatrix}^{\mathrm{T}}\), \(\beta_{0} < \beta_{1}\),
\(\beta_{2} > 0\).  This model is closely related to the Hill equation
or the so called \(E_{max}\)-model \citep{ReeveTurner:2013a}.  In this
model the parameter \(\beta_{0}\) corresponds to the asymptote
(baseline) for \(t \rightarrow - \infty\).  Also here \(\beta_{1}\) is
the asymptote for \(t_j \rightarrow \infty\) and \(\beta_{2}\) is the
strength of the slope.  The last parameter, \(\beta_3\), indicates the
location of the \(ED_{50} = - \beta_3 / \beta_2\), where half of the
saturation is achieved.  In contrast to the four parameter logistic
model for binary data, where \(\mu_{j}(\bm{\beta})\) is the
probability of response at time point \(j\), here \(\beta_{0}\) and
\(\beta_{1}\) are not restricted to
\(0 \leq \beta_{0} < \beta_{1} \leq 1\).

Note, that in the sigmoid model the $ED_{50}$ corresponds to the
inflection point of the mean score function.  If the $ED_{50}$ is
negative, the shape of the sigmoid and the exponential model yield
similar results in practical situations.  Some examples for the shape
of the mean functions for different values of the parameter{s} are
shown in Figure~\ref{fig:mean}.
\begin{figure}
  \centering
  \includegraphics[width=0.45\textwidth]{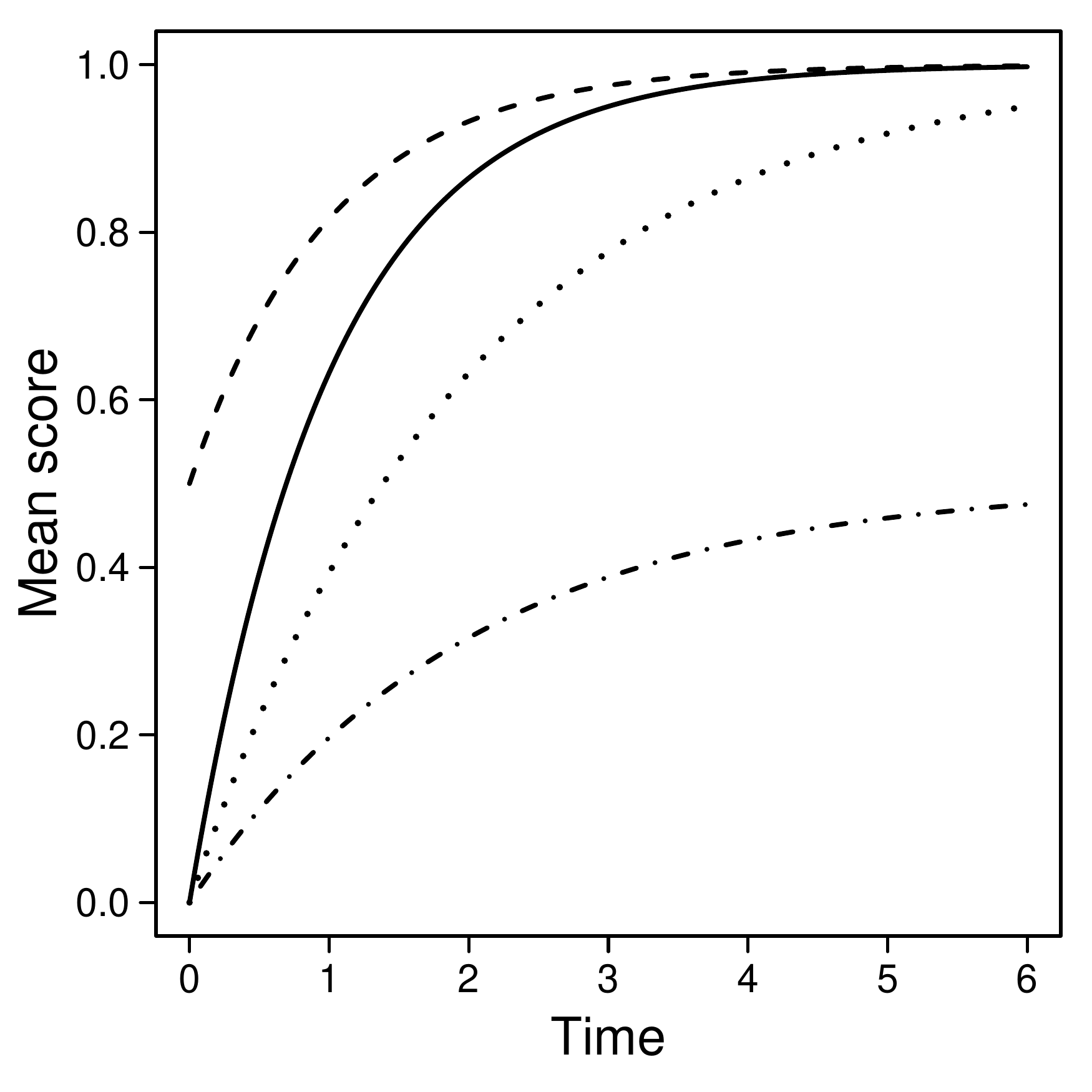}
  \includegraphics[width=0.45\textwidth]{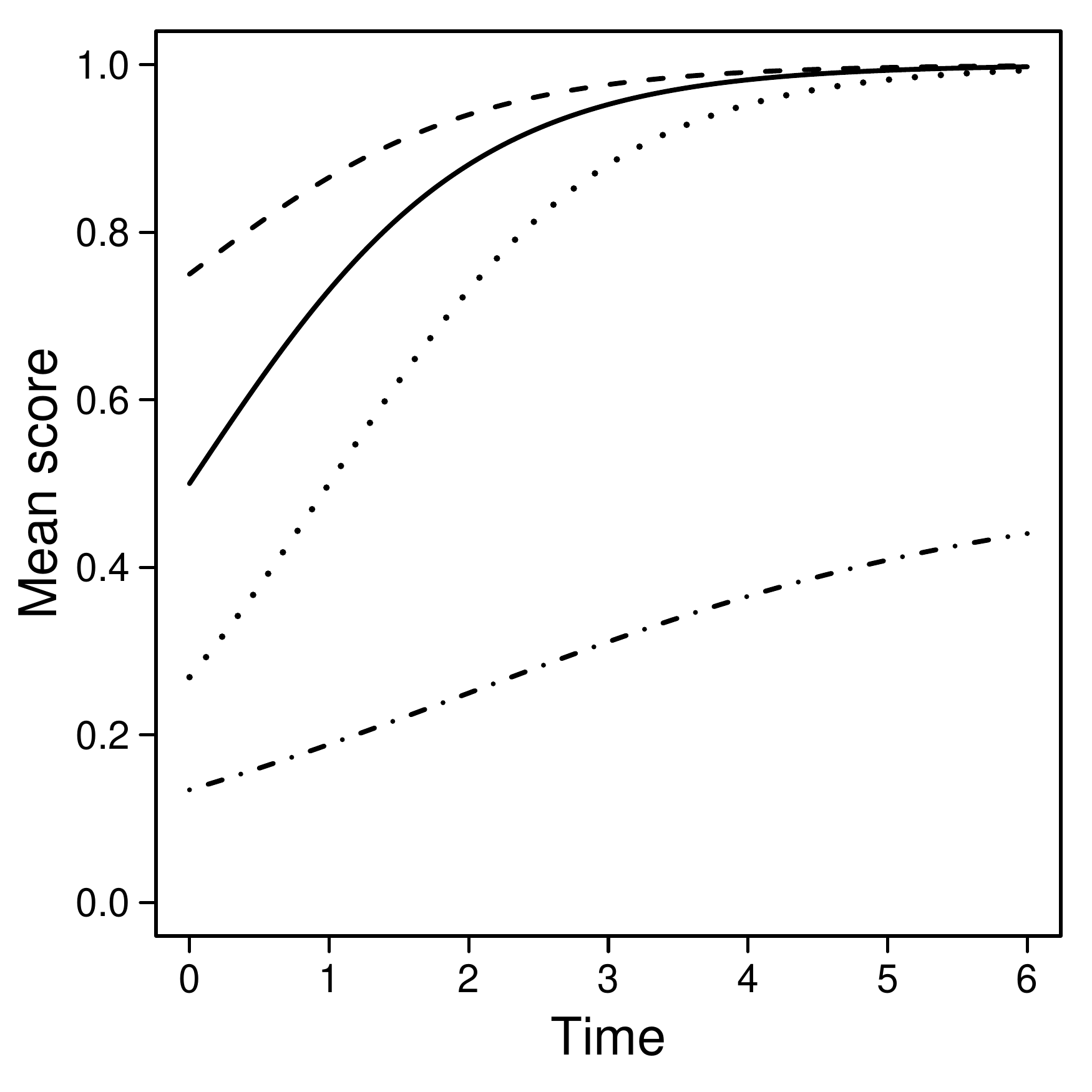}
  \caption{Mean score functions for different values of
    \(\bm{\beta}\). (left: exponential, solid
    $\bm{\beta}=(0\;1\;1)^{\mathrm{T}}$, dashed
    $\bm{\beta}=(0.5\;1\;1)^{\mathrm{T}}$, dotted
    $\bm{\beta}=(0\;1\;0.5)^{\mathrm{T}}$, dash-dotted
    $\bm{\beta}=(0\;0.5\;0.5)^{\mathrm{T}}$; right: 
    sigmoid logistic, solid
    $\bm{\beta} =(0\;1\;1\;0)^{\mathrm{T}} $, dashed
    $\bm{\beta}=(0.5\;1\;1\;0)^{\mathrm{T}}$, dotted
    $\bm{\beta}=(0\;1\;1\;-1)^{\mathrm{T}}$, dash-dotted
    $\bm{\beta}=(0\;0.5\;0.5\;-1)^{\mathrm{T}}$ )}
  \label{fig:mean}
\end{figure}

The parameter vector \(\bm{\beta}\) will commonly be estimated by the
maximum-likelihood estimator \(\hat{\bm{\beta}}\).  To allow for
estimability of \(\bm{\beta}\), the number of time points must be at
least as large as the number of parameters (\(J \geq p\)).  When the
covariance matrix \(\mathbf{V}\) is known, the log-likelihood for
estimating \(\bm{\beta}\) is given by
\begin{equation*}
  l(\bm{\beta})
  = - {\textstyle{\frac{NI}{2}}} \log\left(2 \pi\right)
    - {\textstyle{\frac{N}{2}}} \log\left(\det(\mathbf{V})\right) 
   - {\textstyle{\frac{1}{2}}} \sum_{n = 1}^{N} \left(\mathbf{Y}_n
    - \mathbf{F} \bm{\mu}(\bm{\beta})\right)^{\mathrm{T}}
    \mathbf{V}^{-1}
    \left(\mathbf{Y}_n
    - \mathbf{F} \bm{\mu}(\bm{\beta})\right).
\end{equation*}
The maximum-likelihood estimator \(\hat{\bm{\beta}}\) of
\(\bm{\beta}\) is obtained by equating the derivative of the
log-likelihood to zero.  Only the last expression on the right hand
side depends on $\bm{\beta}$ and, hence, the derivative of the
log-likelihood with respect to $\bm{\beta}$ is
\begin{equation*}
  \frac{\partial l(\bm{\beta})}{\partial \bm{\beta}}
  = \sum_{n = 1}^{N} \mathbf{A}_{\bm{\beta}}^{\mathrm{T}} \mathbf{F}^{\mathrm{T}}
  \mathbf{V}^{-1}
  \left(\mathbf{Y}_n
    - \mathbf{F}
    \bm{\mu}(\bm{\beta})
  \right),
\end{equation*}
where
\begin{equation*}
  \mathbf{A}_{\bm{\beta}}
  =
  \begin{pmatrix}
    \frac{\partial \mu_{1}(\bm{\beta})}{\partial \beta_{0}}
    & \ldots & \frac{\partial \mu_{1}(\bm{\beta})}{\partial
      \beta_{p - 1}} \\
    \vdots & & \vdots \\
    \frac{\partial \mu_{J}(\bm{\beta})}{\partial \beta_{0}}
    & \ldots & \frac{\partial \mu_{J}(\bm{\beta})}{\partial \beta_{p - 1}} \\
  \end{pmatrix}
\end{equation*}
denotes the Jacobian matrix of \(\bm{\mu}(\bm{\beta})\), i.e.\ the
\(J \times p\) matrix of the first derivatives with respect to
\(\bm{\beta}\).

The quality of the maximum-likelihood estimator is measured by the
Fisher information matrix \(\mathbf{M}_{\bm{\beta}}\) which is
proportional to the inverse of the asymptotic covariance matrix of
\(\hat{\bm{\beta}}\), as long as \(\bm{\beta}\) is estimable, i.e.\
\(\mathbf{A}_{\bm{\beta}}\) has full column rank \(p\) and is in the
span of \(\mathbf{F}^{\mathrm{T}}\).  By the definition of the Fisher
information for all observations the whole experiment yields
\begin{equation}
  \label{eq:informationgeneral}
  \mathbf{M}_{\bm{\beta}}
  = \Cov\left(\frac{\partial l(\bm{\beta})}{\partial
      \bm{\beta}}\right)
  = N \mathbf{A}_{\bm{\beta}}^{\mathrm{T}}
  \mathbf{F}^{\mathrm{T}} \mathbf{V}^{-1} \mathbf{F} 
  \mathbf{A}_{\bm{\beta}} \, .
\end{equation}

When the number \(I_j\) of observations is greater than \(0\) for each
time point \(t_j\), we can make use of the
representation~\eqref{eq:covoftheta2} in the unstructured situation of
Subsection~\ref{sec:sub:unstructured}.  With this representation the
information matrix simplifies to
\begin{equation}
  \mathbf{M}_{\bm{\beta}}
  = N \mathbf{A}_{\bm{\beta}}^{\mathrm{T}}
  \left(
  \sigma_{\varepsilon}^{2} \mathbf{M}_{0}^{-1}
  + \mathbf{\Sigma}_{\gamma} 
  \right)^{-1} 
  \mathbf{A}_{\bm{\beta}} ,
  \label{eq:InformationMatrix}
\end{equation}
where
\(\mathbf{M}_{0} = \mathbf{F}^{\mathrm{T}} \mathbf{F} =
\mathrm{diag}_{j=1, \ldots, J}(I_j)\) denotes the corresponding
information matrix in the linear fixed effects model of a one-way
layout with time points \(t_j\) as levels.

In general, when there may be time points \(t_j\) without observations
(\(I_j=0\)), we can give a more general representation of the core
expression \(\mathbf{F}^{\mathrm{T}} \mathbf{V}^{-1} \mathbf{F}\) in
the Fisher information matrix~\eqref{eq:informationgeneral}.

\begin{lemma}
  \label{lem:infogeneralsupport}
  Denote by
  \(\mathbf{M}_{0}^{1/2} = \mathrm{diag}_{j=1, \ldots,
    J}(\sqrt{I_j})\).  Then
  \begin{equation}
    \label{eq:infogeneralsupport}
    \mathbf{F}^{\mathrm{T}} \mathbf{V}^{-1} \mathbf{F}
    = \mathbf{M}_{0}^{1/2}
    \left(
      \sigma_{\varepsilon}^{2} \mathbf{I}_{J}
      + \mathbf{M}_{0}^{1/2} \mathbf{\Sigma}_{\gamma} \mathbf{M}_{0}^{1/2}
    \right)^{-1} 
    \mathbf{M}_{0}^{1/2} .
  \end{equation}
\end{lemma}

The proof is given in the Appendix.  By
Lemma~\ref{lem:infogeneralsupport}, the Fisher information matrix has
the form
\begin{equation}
  \mathbf{M}_{\bm{\beta}}
  = N \mathbf{A}_{\bm{\beta}}^{\mathrm{T}} \mathbf{M}_{0}^{1/2}
  \left(
  \sigma_{\varepsilon}^{2} \mathbf{I}_{J}
  + \mathbf{M}_{0}^{1/2} \mathbf{\Sigma}_{\gamma} \mathbf{M}_{0}^{1/2}
  \right)^{-1} 
  \mathbf{M}_{0}^{1/2} \mathbf{A}_{\bm{\beta}} .
  \label{eq:InformationMatrixGeneral}
\end{equation}
This Fisher information is non-singular if \(\bm{\beta}\) is
estimable.

In the particular case of a linearly parameterized mean response, like
the straight line growth curve, the Jacobian
\(\mathbf{A}_{\bm{\beta}} = \mathbf{A}\) does not depend on
\(\bm{\beta}\) as \(\bm{\mu}(\bm{\beta}) = \mathbf{A}\bm{\beta}\).  In
this case, for a subject \(n\), the mixed effects model can be written
as
\(\mathbf{Y}_n = \mathbf{F} \mathbf{A} \bm{\beta} + \mathbf{F}
\bm{\gamma}_n + \bm{\varepsilon}_n\).  Then \(\mathbf{F} \mathbf{A}\)
is the design matrix of the fixed effects part on the individual
level.  If \(\bm{\beta}\) is estimable, the maximum-likelihood
estimator is unique and coincides with the generalized least squares
estimator
\(\hat{\bm{\beta}}_{GLS} = \frac{1}{N} \sum_{n = 1}^{N}
\left(\mathbf{A}^{\mathrm{T}} \mathbf{F}^{\mathrm{T}} \mathbf{V}^{-1}
  \mathbf{F} \mathbf{A}\right)^{-1} \mathbf{A}^{\mathrm{T}}
\mathbf{F}^{\mathrm{T}} \mathbf{V}^{-1} \mathbf{Y}_n\).  As in the
unstructured situation of the previous subsection, the generalized
least square estimator is equal to the estimator
\(\left(\mathbf{A}^{\mathrm{T}} \mathbf{F}^{\mathrm{T}}
  \mathbf{V}^{-1} \mathbf{F} \mathbf{A}\right)^{-1}
\mathbf{A}^{\mathrm{T}} \mathbf{F}^{\mathrm{T}} \mathbf{V}^{-1}
\bar{\mathbf{Y}}\) for an average individual
\(\bar{\mathbf{Y}} = \frac{1}{N} \sum_{n = 1}^{N} \mathbf{Y}_n\) as
well as to the average
\(\frac{1}{N} \sum_{n = 1}^{N} \hat{\bm{\beta}}_{n,GLS}\) of
individual fits
\(\hat{\bm{\beta}}_{n,GLS} = \left(\mathbf{A}^{\mathrm{T}}
  \mathbf{F}^{\mathrm{T}} \mathbf{V}^{-1} \mathbf{F}
  \mathbf{A}\right)^{-1} \mathbf{A}^{\mathrm{T}}
\mathbf{F}^{\mathrm{T}} \mathbf{V}^{-1} \mathbf{Y}_n\) based solely on
the observations of single subjects \(n\).  However, in contrast to
the unstructured case, the generalized least squares estimator differs
in general from the ordinary least squares estimator, which can be
seen from the example of the ratio model below.

For the exponential model, the derivatives of
\(\mathbf{A}_{\bm{\beta}}\) with respect to the {parameters}
\(\beta_{0}\), \(\beta_{1}\), and \(\beta_{2}\) are
\begin{equation}
  \label{eq:derive3PExp}
  \frac{\partial \mu_{j}(\bm{\beta})}{\partial \beta_{0}}
  = \exp( - \beta_{2} t_{j}),
  \quad \frac{\partial   \mu_{j}(\bm{\beta})}{\partial \beta_{1}}
  = 1 - \frac{\partial \mu_{j}(\bm{\beta})}{\partial \beta_{0}},
  \quad\text{and}\quad
  \frac{\partial \mu_{j}(\bm{\beta})}{\partial \beta_{2}}
  = t_j (\beta_{1} - \beta_{0})
  \frac{\partial \mu_{j}(\bm{\beta})}{\partial \beta_{0}} \, .
\end{equation}
The corresponding derivatives for the logistic model are
\begin{align}
  \label{eq:derive4PL}
  \frac{\partial \mu_{j}(\bm{\beta})}{\partial \beta_{0}}
  & = \frac{1}{1 + \exp(\beta_{2} t_{j} + \beta_{3})},
    \qquad \frac{\partial \mu_{j} (\bm{\beta})}{\partial \beta_{1}} 
    = 1 - \frac{\partial \mu_{j}(\bm{\beta})}{\partial \beta_{0}}
    ,\\ \nonumber
  \frac{\partial \mu_{j}(\bm{\beta})}{\partial \beta_{2}}
  & = t_{j} \frac{\partial \mu_{j}(\bm{\beta})}{\partial \beta_{3}} ,
    \quad \text{and} \quad 
    \frac{\partial \mu_{j}(\bm{\beta})}{\partial \beta_{3}}
    = (\beta_{1} - \beta_{0})
    \frac{\exp(\beta_{2} t_{j} + \beta_{3})}{(1 + \exp(\beta_{2} t_{j} + \beta_{3}))^2} \, . 
\end{align}
Note that in both of these two model specifications the derivatives
depend less on the specific values for \(\beta_0\) for the baseline
and \(\beta_1\) for the saturation level themselves, but on their
difference \(\beta_1 - \beta_0\).  In the particular case of a linear
model also the Fisher information
\(\mathbf{M}_{\bm{\beta}}=\mathbf{M}\) does not depend on the vector
\(\bm{\beta}\) of location parameters.

\section{Optimal Designs}
\label{sec:optimal-design}
A design, i.e.\ a plan how to conduct an experiment, usually has two
parts: Firstly we need to know under which conditions to observe,
which in our case corresponds to the time points \(t_j\).  The second
part of a design determines, how many observations should be spent at
each time point.  As mentioned at the beginning of the preceding
section, the time points \(t_{j} = j - 1\) are fixed, i.e.\ we need
not optimize with respect to the time points.  Thus, in the present
situation, a design can be described by the numbers \(I_{j}\) of
observations at each time point under the constraint that the total
number \(\sum_{j = 1}^{J} I_j\) of observations equals \(I\) and will
be denoted by \(\xi =
\begin{pmatrix}
  I_1 & \cdots & I_J
\end{pmatrix}
\).  When \(I_{j} \geq 0\) are integers, the design is called an exact
design.  Actually, we may allow one or more of the numbers \(I_j\) to
be equal to zero, which means that no observations are made at the
corresponding time points.  These kind of designs can be used
directly.  However, to optimize integers \(I_j\) is a discrete
optimization problem which may be hard to solve.  If \(I_{j} \geq 0\)
are relaxed to be real numbers, the design \(\xi\) is called an
approximate design \citep[see][]{Kiefer:1959}.  The latter is
appealing, because it is possible to use methods from convex
optimization to optimize with respect to the design.  While this
yields benchmark designs, rounding of the real numbers \(I_{j}\) to
integers may become necessary for using these designs in applications.

For an approximate design \(\xi\), the standardized (per unit)
information matrix
\(\mathbf{M}_{\bm{\beta}}(\xi) = \frac{1}{N} \mathbf{M}_{\bm{\beta}}\)
is defined by the representations \eqref{eq:InformationMatrix} and
\eqref{eq:InformationMatrixGeneral}, respectively.  While for a linear
model of the mean score, like the unstructured model of
Subsection~\ref{sec:sub:unstructured} or the straight line growth
curve model, the information matrix \(\mathbf{M}_{\bm{\beta}}(\xi)\)
does not depend on the value of the parameter vector \(\bm{\beta}\),
such a dependence will typically occur when the mean score is a
nonlinear function over time.

To measure the performance of a design, we will use the
\(D\)-criterion which is invariant with respect to reparameterization
and scaling and which aims at minimizing the volume of the asymptotic
confidence ellipsoid for estimating \(\bm{\beta}\).  Then a design
\(\xi^* = \begin{pmatrix} I_1^* & \cdots & I_J^*
\end{pmatrix}\) will be called locally \(D\)-optimal for a specific
value of \(\bm{\beta}\) if
\(\log\det(\mathbf{M}_{\bm{\beta}}(\xi^*)) \geq
\log\det(\mathbf{M}_{\bm{\beta}}(\xi))\) for all \(\xi\) with
non-singular information matrix such that \(\bm{\beta}\) is estimable,
i.e.\ for approximate designs, \(\mathbf{A}_{\bm{\beta}}\) has full
column rank \(p\) and is in the span of the standardized (per unit)
information matrix
\(\mathbf{M}_{0}(\xi) = \mathrm{diag}_{j = 1, \ldots, J}(I_j)\) for
\(\xi\) in the corresponding one-way layout model.

\begin{lemma}
  \label{lem:concave}
  The \(D\)-criterion \(\log\det(\mathbf{M}_{\bm{\beta}}(\xi))\) is a
  continuous, concave function of \(\xi\) on the set of all designs
  for which \(\bm{\beta}\) is estimable.
\end{lemma}

\begin{proof}
  The continuity can be seen from the general representation of the
  Fisher information matrix in~\eqref{eq:InformationMatrixGeneral} .
  For fully supported designs, which take \(I_j > 0\) observations at
  each of the \(J\) time points \(j=1, \ldots, J\), the concavity
  follows directly from Lemma~2 in \citet{Schmelter:2007}.  The
  concavity extends to all designs for which \(\bm{\beta}\) is
  estimable by the continuity of the criterion.
\end{proof}

For a given design \(\xi\), its performance can be compared to the
optimal design \(\xi^*\) in terms of the efficiency which is defined
by
\begin{equation*}
  \left(
    \frac{\det(\mathbf{M}_{\bm{\beta}}(\xi))}{\det(\mathbf{M}_{\bm{\beta}}(\xi^*))}
  \right)^{1/p} .
\end{equation*}
Here \(\det(\mathbf{M})^{1/p}\) is the homogeneous version of the
\(D\)-criterion such that the efficiency can be interpreted as the
proportion of units needed when the optimal design \(\xi^*\) is used
to obtain the same precision as for the design \(\xi\) under
consideration.

\subsection{The Case \(J = p\)}
\label{sec:sub:saturated}
If the number \(J\) of time points is equal to the number \(p\) of
parameters, then $\mathbf{A}_{\bm{\beta}}$ is a quadratic matrix and
consequently
\begin{equation}
  \label{eq:detpeqJ}
  \log\det(\mathbf{M}_{\bm{\beta}}(\xi))
  = - \log\det\left(\sigma_{\varepsilon}^{2}\mathbf{M}_{0}(\xi)^{-1} + \mathbf{\Sigma}_{\gamma}\right) 
  + c ,
\end{equation}
where \(c\) denotes a generic constant not depending on \(\xi\).
This covers also the situation of unstructured time dependence 
of Subsection~\ref{sec:sub:unstructured} when \(\mathbf{A}_{\bm{\beta}} = \mathbf{I}_J\).

Note that, in general, the matrix \(\mathbf{A}_{\bm{\beta}}\) has to
be non-singular in order to allow for estimability of the parameter
vector.  Since the time points are fixed, the determinant of the
Jacobian matrix \(\mathbf{A}_{\bm{\beta}}\) is constant with respect
to the design \(\xi\).  Hence, the locally \(D\)-optimal design
\(\xi^*\) does not depend on $\bm{\beta}$.  In particular, the design
\(\xi^*\) maximizes \(\log\det(\mathbf{M}_{\bm{\beta}}(\xi))\) if it
minimizes \(\log\det(\Cov_{\xi}(\hat{\bm{\mu}}))\), i.e.\ it coincides
with the \(D\)-optimal design to estimate $\bm{\mu}$ in the case of
the unstructured model in Subsection~\ref{sec:sub:unstructured},
independent of the parameterization
\(\bm{\mu} = \bm{\mu}(\bm{\beta})\).

In the case of compound symmetry, the covariance matrix
\(\mathbf{\Sigma}_{\gamma,CS}\) is invariant with respect to
permutations of the time points.  This is also reflected in the
criterion
\begin{align*}
  \log\det(\mathbf{M}_{\bm{\beta}}(\xi))
  & = - \log\det\left(
    {\textstyle{\diag_{j = 1, \ldots, J}}}\left(1 - \rho + {\textstyle{\frac{1}{\tau^{2} I_{j}}}}\right)
    + \rho \mathbf{1}_{J} \mathbf{1}_{J}^{\mathrm{T}}
    \right) + c \\
  &
    =
    - \sum_{j = 1}^{J} \log\left(1 - \rho + {\textstyle{\frac{1}{\tau^2 I_{j}}}}\right)
    - \log\left(1 + \rho \sum_{j = 1}^{J} \left(1 - \rho + {\textstyle{\frac{1}{\tau^2 I_j}}}\right)^{-1}\right) + c
\end{align*}
to be optimized which is also left unchanged by permutations of the
entries \(I_{j}\) in the design.  As a consequence of the concavity,
the optimal design \(\xi^*\) should be as balanced as possible, which
in terms of the \(I_{j}^*\) can be written as
\(|I_{j}^* - I_{j'}^*| \leq 1\), \(j \neq j'\).  More formally, denote
by \([a]\) the integer part of \(a\) and by \(m \bmod n = m - n[m/n]\)
the remainder of \(m\) with respect to division by \(n\) for integers
\(m\) and \(n\).
\begin{lemma}
  \label{lem:compsymoptdes2}
  Let \(J = p\) and the covariance structure be compound symmetry,
  then an exact design is \(D\)-optimal if \(I_j^*=[I/J]+1\) for
  \(J'=I \bmod J\) time points and \(I_j^*=[I/J]\) for \(J-J'\) time
  points.
  
  In particular, if \(I\) is a multiple of \(J\), then the uniform
  design is \(D\)-optimal which assigns \(I_J^* = I/J\) items to each
  of the time points \(j=1, \ldots, J\).
\end{lemma}

This result can be obtained by standard arguments from the concavity
of the \(D\)-criterion and the exchangeability of the time points
under compound symmetry.

By the strict concavity of the \(D\)-criterion, we obtain for the
approximate design

\begin{lemma}
  \label{lem:compsymoptdes1}
  Let \(J = p\) and the covariance structure be compound symmetry,
  then
  \begin{equation*}
    \xi^* =
    \begin{pmatrix}
      \frac{I}{J} & \cdots & \frac{I}{J}
    \end{pmatrix}
  \end{equation*}
  is the \(D\)-optimal approximate design.
\end{lemma}

In the case of autoregressive covariance, there is less symmetry
structure available in the covariance matrix
\(\mathbf{\Sigma}_{\gamma,AR}\).  Only time reversal does not alter
the covariance matrix and, hence, the determinant.  For time reversal,
the frequencies \(I_{j}\) and \(I_{J - j + 1}\) have to be
interchanged within each symmetric pair of time points
\((j, J - j + 1)\), \(1 \leq j \leq J / 2\).
\begin{lemma}
  \label{lem:ar-symmetry}
  Let \(J = p\) and the covariance structure be autoregressive, then
  the optimal approximate design is symmetric in time, i.e.\ it is of
  the form
  \begin{equation*}
    \xi^* =
    \begin{pmatrix}
      I_{1}^* & I_{2}^* & \ldots & I_{J/2}^* & I_{J/2+1}^* & \cdots & I_{2}^* & I_{1}^*
    \end{pmatrix} 
  \end{equation*}
  if \(J\) is even, and
  \begin{equation*}
    \xi^* =
    \begin{pmatrix}
      I_{1}^* & I_{2}^* & \cdots & I_{(J+1)/2}^* & \cdots & I_{2}^* & I_{1}^*
    \end{pmatrix}
\end{equation*}
if \(J\) is odd.\end{lemma}

Also this result follows by concavity and invariance of the
\(D\)-criterion.

In the particular case of \(J=3\) time points, under the assumption of
an autoregressive correlation structure, the optimal design has the
form
\(\xi^* = \begin{pmatrix} I_{1}^* & I_{2}^* & I_{1}^* \end{pmatrix}\),
where \(I_2^*=I-2I_1^*\).  Hence, optimization has only to be done
with respect to the single variable \(I_1\).  Moreover, in terms of
weights \(w_j = I_j/I\), the optimal weight \(w_1^*\) depends on the
sample size \(I\) and the variance components
\(\sigma_{\varepsilon}^2\) and \(\sigma_{\gamma}^2\) only through the
standardized variance ratio
\(a=I\tau^2=I\sigma_{\gamma}^2/\sigma_{\varepsilon}^2\).  The optimal
weight \(w_1^*\) can then be obtained in dependence on the
standardized variance ratio \(a\) and the correlation \(\rho\) as the
root of the cubic polynomial
\((a^2 w (1 - w) + a + 1) (1 - 3w) - a^2 \rho^2 w (1 - 2w)^2 + a^2
\rho^4 w^3\) in \(0 < w \leq 1/2\).  Note that for the uncorrelated
case (\(\rho=0\)) and for the case of a constant subject effect
(\(\rho=1\)) the optimal weight is \(w_1^*=1/3\) in accordance with
the situation of compound symmetry.  For fixed standardized variance
ratios of \(a=50\), \(100\) and \(200\), the optimal weight \(w_1^*\)
is shown in Figure~\ref{fig:xy-ar1-j3} in dependence on the
correlation \(\rho\).  From this figure it can be seen that the
optimal weight \(w_1^*\) is slightly descending in the correlation for
small to moderate values of \(\rho\) and then ascending back to
\(1/3\) for large values of the correlation \(\rho\).  However, the
gain of the optimal design is not substantial here as, for \(a=50\),
\(100\) and \(200\), the uniform design (\(w=1/3\)) has efficiency of
more than \(99.5\%\).

\begin{figure}
  \centering
  \includegraphics[width=0.45\textwidth]{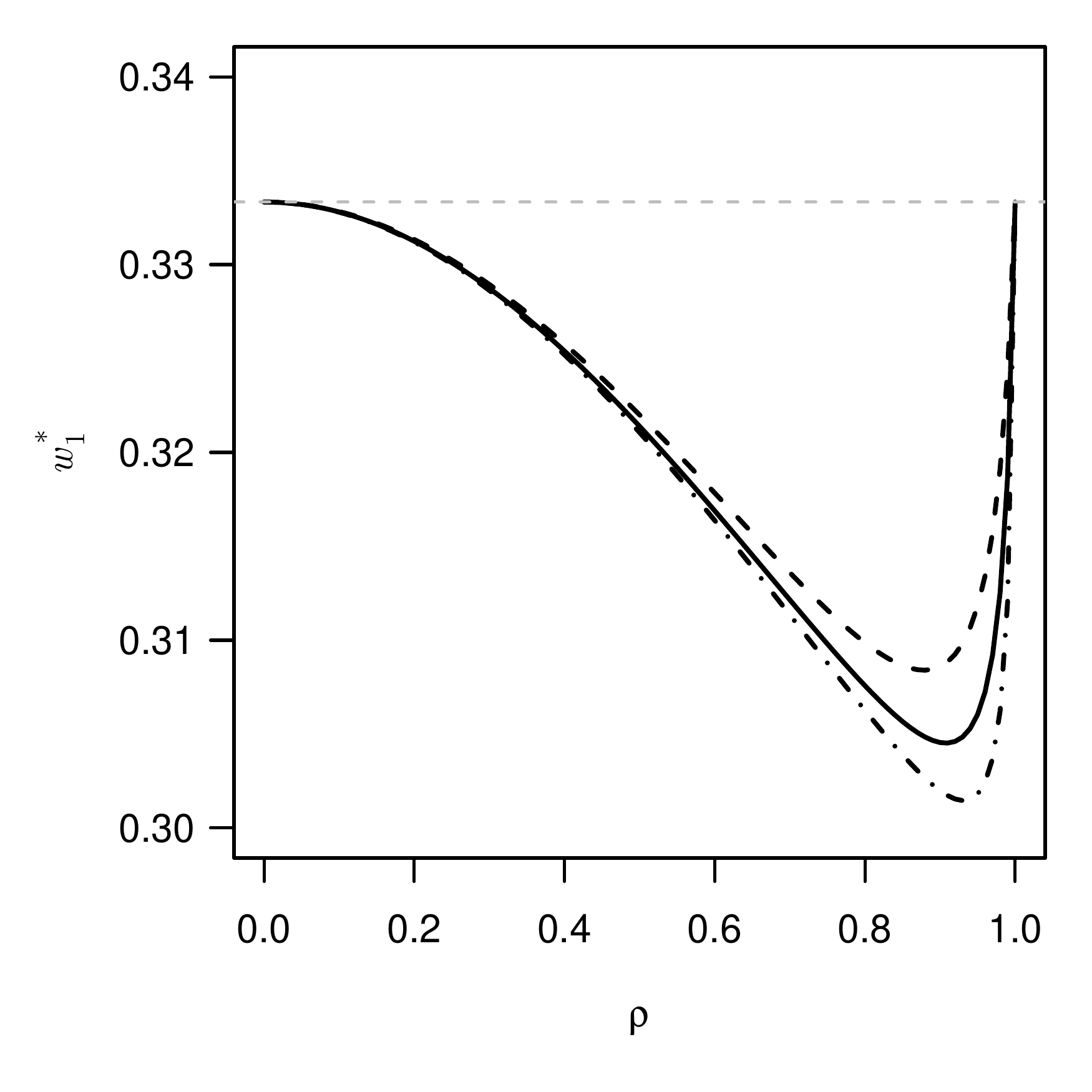}
  \caption{Optimal weight \(w_1^*=I_1^*/I\) with respect to \(\rho\)
    for the straight line model with autoregressive covariance
    structure \(J = p = 3\) and different \(a=I\tau^2\). (dashed
    \(a=50\), solid \(a=100\), dash-dotted \(a=200\)) The dashed
    horizontal line marks the weight \(1 / 3\) of the uniform
    design.}
  \label{fig:xy-ar1-j3}
\end{figure}
\begin{figure}
  \centering
  \includegraphics[width=0.45\textwidth]{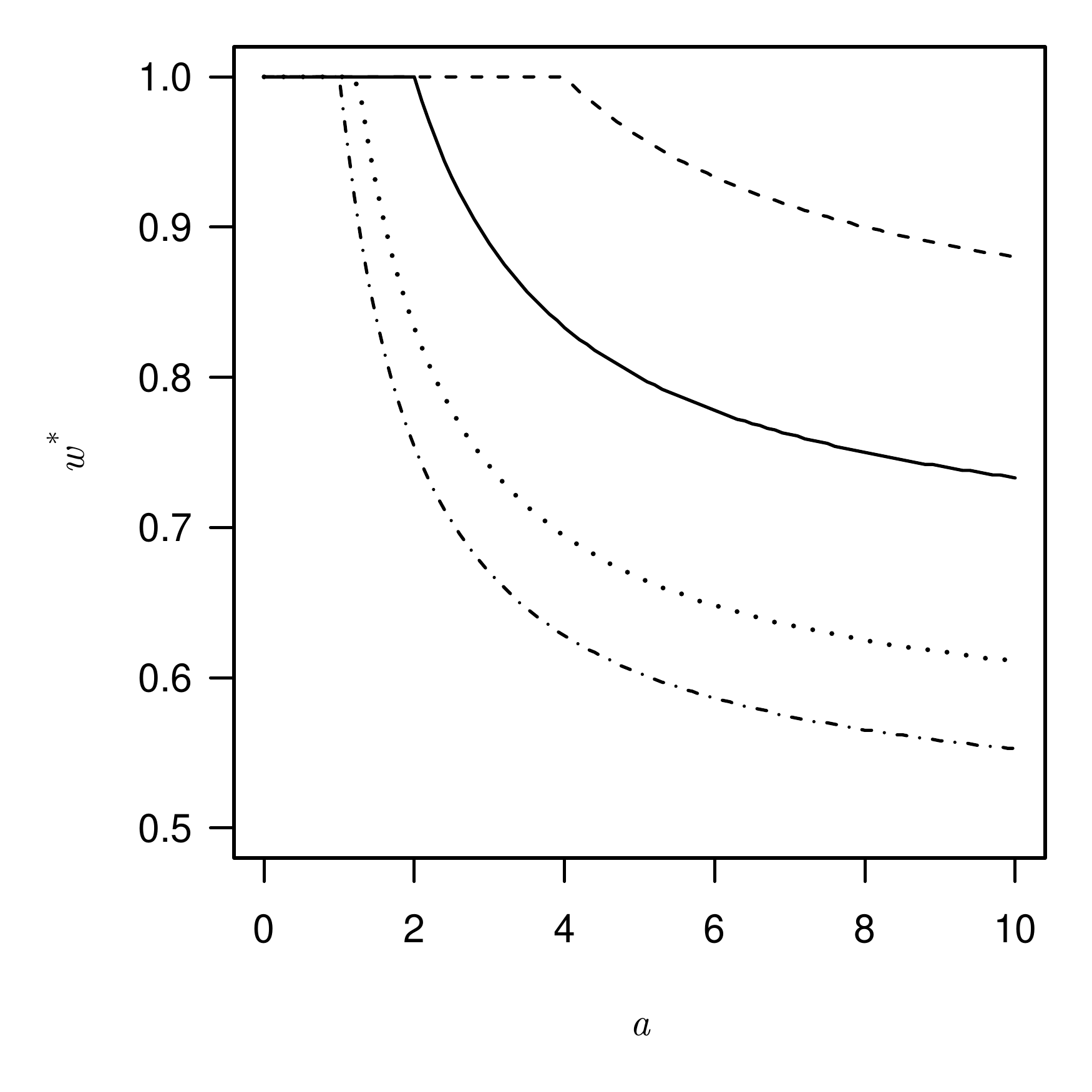}
  \caption{Optimal weight \(w^*=I_2^*/I\) with respect to
    \(a=I\tau^2\) for the ratio model (\(J = 2\), \(p = 1\)) for different
    values of \(\rho\) (dashed \(\rho=0.25\), solid \(\rho=0.5\),
    dotted \(\rho=0.8\), dash-dotted \(\rho = 0.99\)).}
  \label{fig:ratio-opt-weight}
\end{figure}
\begin{figure}
  \centering
    \includegraphics[width=0.45\textwidth]{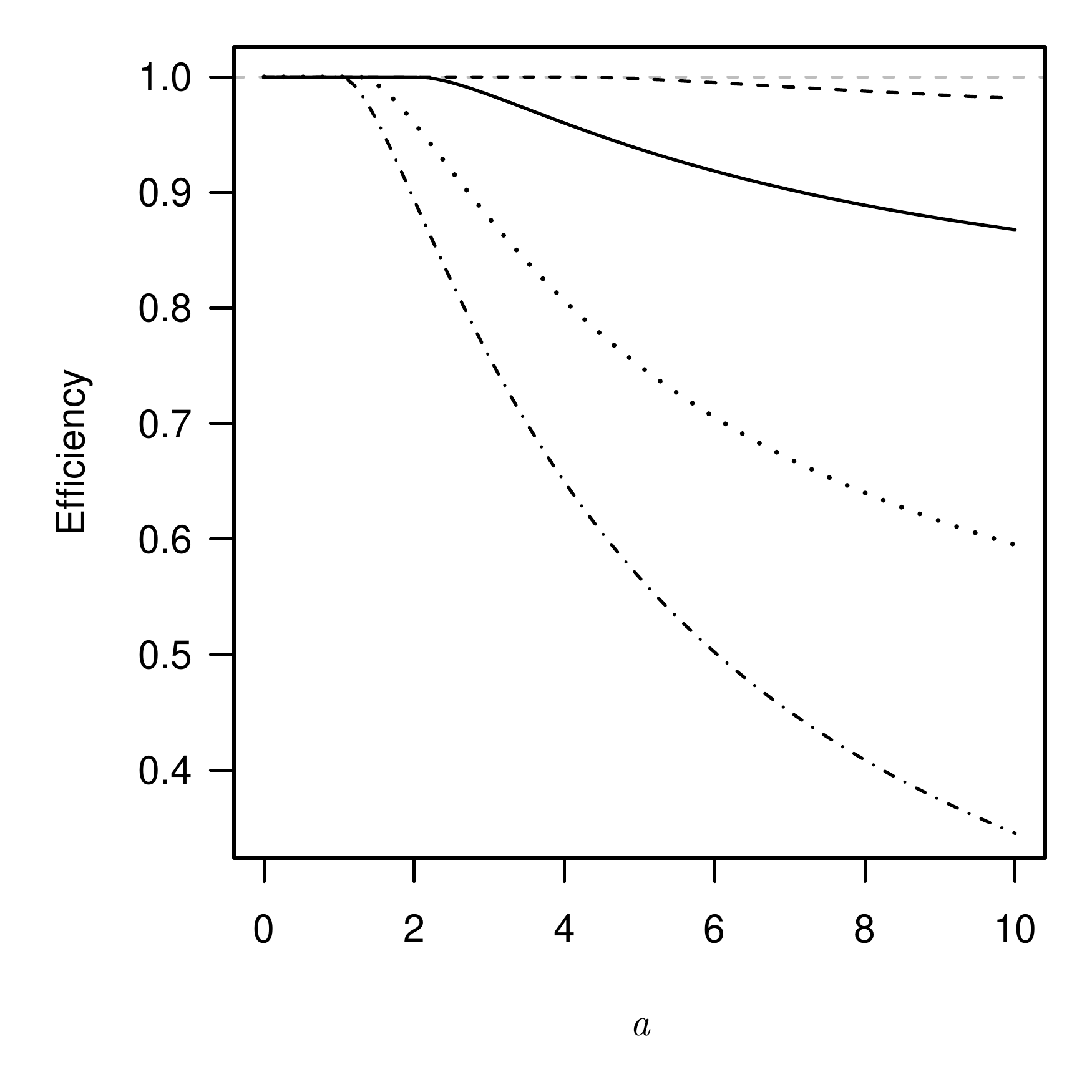}
   \includegraphics[width=0.45\textwidth]{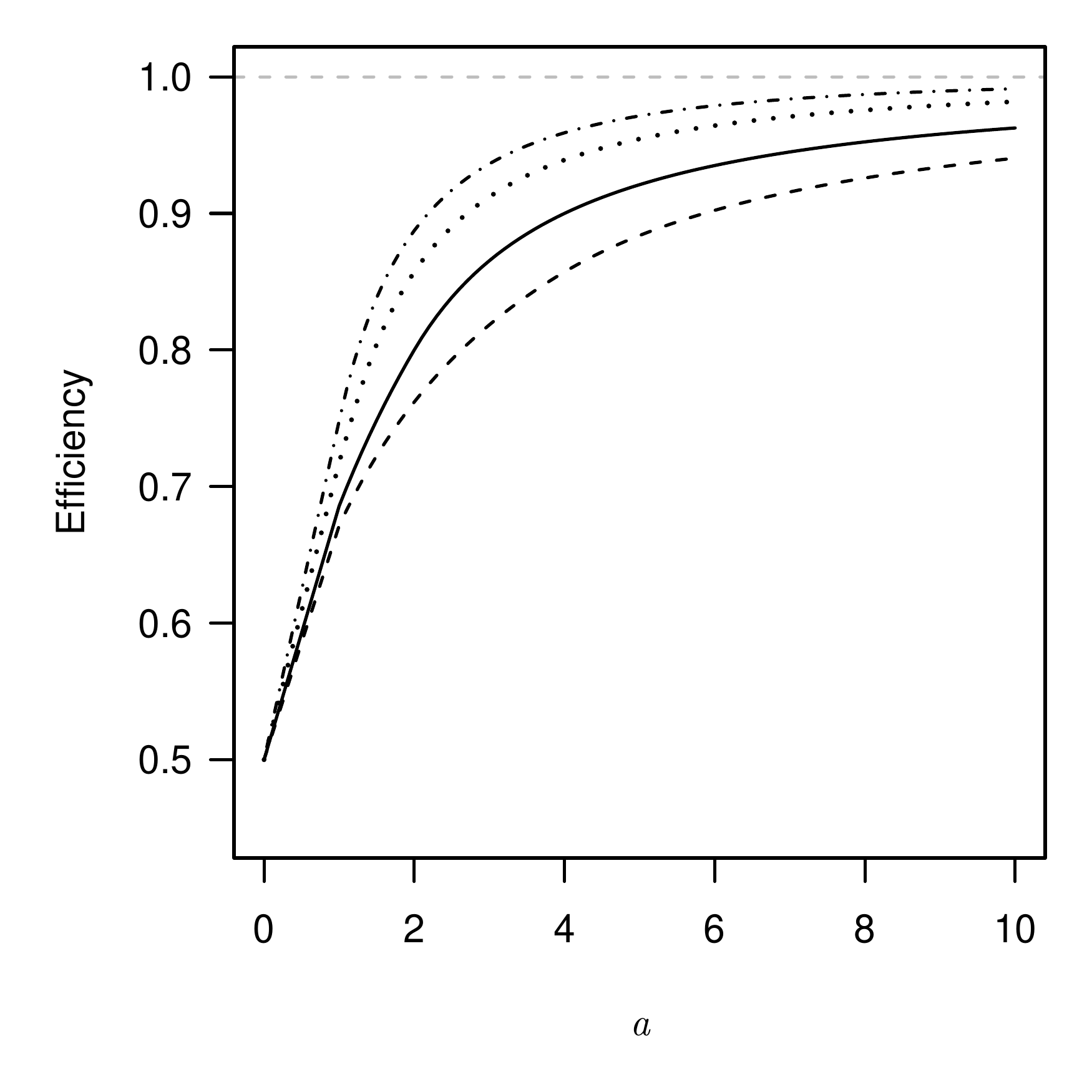}
    \caption{Efficiency of the minimally supported design 
      (\(w=1\); left) and of the uniform design
      (\(w=1/2\); right) with respect 
      to \(a=I\tau^2\) for
      the ratio model (\(J = 2\), \(p = 1\)) for different values of
      \(\rho\) (dashed \(\rho=0.25\), solid
      \(\rho=0.5\), dotted \(\rho=0.8\), dash-dotted \(\rho = 0.99\)).}
  \label{fig:ratio-eff}
\end{figure}

\subsection{The Case \(J > p\)}
\label{sec:sub:Jlargerp}
When the number \(J\) of time points is larger than the number \(p\)
of parameters, estimability of the parameter vector \(\bm{\beta}\)
requires that \(\mathbf{A}_{\bm{\beta}}\) has full column rank \(p\)
and that observations are made at, at least, \(p\) different time
points, i.e.\ \(I_j > 0\) for, at least, \(p\) distinct time points
\(t_j\).  In this case, optimal designs may be restricted to a
proper subset of time points.  However, as we will see, fully
supported designs, which take \(I_j > 0\) observations at each of the
\(J\) time points \(j=1, \ldots, J\), will play an important role.

For illustrative purposes we start here with two linear models for the
mean response: The straight line growth model and, even
simpler, the ratio model.

In the ratio model the mean response is fitted as a straight line
\(\mu_{j}(\bm{\beta}) = \beta_{1} t_{j}\) through the offset
\(t_1=0\). Here \(p=1\), and we consider only \(J = 2\) time points
\(t_1=0\), \(t_2=1\) such that the correlation structure reduces to
\(\mathbf{\Sigma}_{\gamma} = \begin{pmatrix} 1 & \rho
  \\
  \rho & 1
\end{pmatrix}\) in both the compound symmetry and the autoregressive
case.  For uncorrelated time points (\(\rho = 0\)), observations at
\(t_1 = 0\) do not bear any information on the parameter
\(\bm{\beta} = \beta_1\).  Then the optimal design \(\xi^*\) is
minimally supported and assigns all items to the second time point
\(t_2 = 1\) as in the corresponding fixed effects model without random
effects.  However, for correlated time points (\(\rho > 0\)), the
situation changes.  Observed response at the offset \(t_1 = 0\) can be
used to better estimate the ratio \(\beta_{1}\) and a fully supported
design becomes optimal if the standardized variance ratio
\(a = I \tau^2 = I \sigma_{\gamma}^2 / \sigma_{\varepsilon}^2\) is
sufficiently large: \(a > 1/\rho\).  Then the optimal design \(\xi^*\)
assigns only a proportion
\(I_2^* / I = w^* = (a + 1) / (a + a\rho) < 1\) of items to the
effective time point \(t_2 = 1\) and the remaining proportion
\(1 - w^* = (a\rho - 1) / (a + a\rho) > 0\) to the offset \(t_1 = 0\).
The optimal weights \(w^*\) are shown in
Figure~\ref{fig:ratio-opt-weight} for various values of the
correlation \(\rho\).  For large values of the standardized variance
ratio (\(a \to \infty\)), the optimal weight tends to
\(1 / (1 + \rho) > 0.5\) such that still the majority of items will be
assigned to \(t_2 = 1\).  The gain of the optimal design \(\xi^*\) is
illustrated in Figure~\ref{fig:ratio-eff} where the efficiency of the
optimal design \(\xi_1\) for the fixed effect model (\(w_2 = 1\)) and
the uniform design \(\bar{\xi}\) (\(w_2 = 1/2\)) is shown,
respectively.  The efficiency of \(\xi_1\) decreases to \(1 - \rho^2\)
when the standardized variance ratio \(a\) tends to infinity, while
the efficiency of \(\bar{\xi}\) is increasing from \(0.5\) for
\(a = 0\) to \(1\) for \(a \to \infty\).  Hence, the optimal design
\(\xi^*\) substantially outperforms the minimally supported design
\(\xi_1\) which is optimal in the fixed effects model if \(a\) becomes
large.

A similar phenomenon occurs in the straight line growth model
\(\mu_{j}(\bm{\beta}) = \beta_{0} + \beta_{1} t_{j}\) even in
the case of uncorrelated time points.  Using the symmetry structure of
time reversal as in Lemma~\ref{lem:ar-symmetry}, we see that for
equidistant time points the optimal design \(\xi^*\) is symmetric in
time, i.e.\ it has the form specified in Lemma~\ref{lem:ar-symmetry}
as long as the covariance matrix \(\mathbf{\Sigma}_\gamma\) shares the
symmetry property which holds for both the case of compound symmetry
as well as the case of autoregressive random effects.  In particular,
for \(J = 3\) time points, the optimal design has the form
\(\xi^* = \begin{pmatrix} I_{1}^* & I_{2}^* & I_{1}^* \end{pmatrix}\),
where \(I_2^* = I - 2 I_1^*\).  Hence, optimization has only to be
taken with respect to \(I_1^*\).  For simplification we consider the
situation of uncorrelated random effects which appears as the special
case \(\rho=0\) in either the compound symmetric or the autoregressive
case.  In terms of weights \(w_j = I_j / I\), the optimal weight
\(w_1^*\) also depends only on the standardized variance ratio
\(a = I\tau^2\).  By straightforward calculation the optimal weight
\(w_1^*\) can be determined as the root of the cubic polynomial
\(18 a^2 w^3 - (20 a^2 + 18 a) w^2 + 5 (a^2 + a) w + a + 1\) in
\(0 < w \leq 1/2\) if \(a \geq 2 (\sqrt{2} - 1) \approx 0.8284\),
while \(w_1^*=1/2\) for \(a \leq 2 (\sqrt{2} - 1)\).  The optimal
weight \(w_1^*\) is shown in Figure~\ref{fig:xy-linreg} in dependence
on the standardized variance ratio \(a\).  From this figure it can be
seen that the optimal weight \(w_1^*\) is decreasing in \(a\) for
\(a>2 (\sqrt{2} - 1)\), and it tends to
\((10 - \sqrt{10})/18 \approx 0.3799\) for \(a \to \infty\).  Hence,
it can be observed that the optimal design \(\xi^*\) is minimally
supported on the endpoints of the time interval for small values of
the standardized variance ratio \(a = I \tau^2\), while for larger
values the optimal design is fully supported on all time points.

The efficiency of the minimally supported design (\(w_1 = 1/2\)) and
of the uniform design (\(w_1 = 1/3\)) in Figure~\ref{fig:j3-null-eff}
show a similar behavior as exhibited in Figure~\ref{fig:ratio-eff} for
their counterparts in the ratio model, respectively.
\begin{figure}
  \centering
  \includegraphics[width=0.45\textwidth]{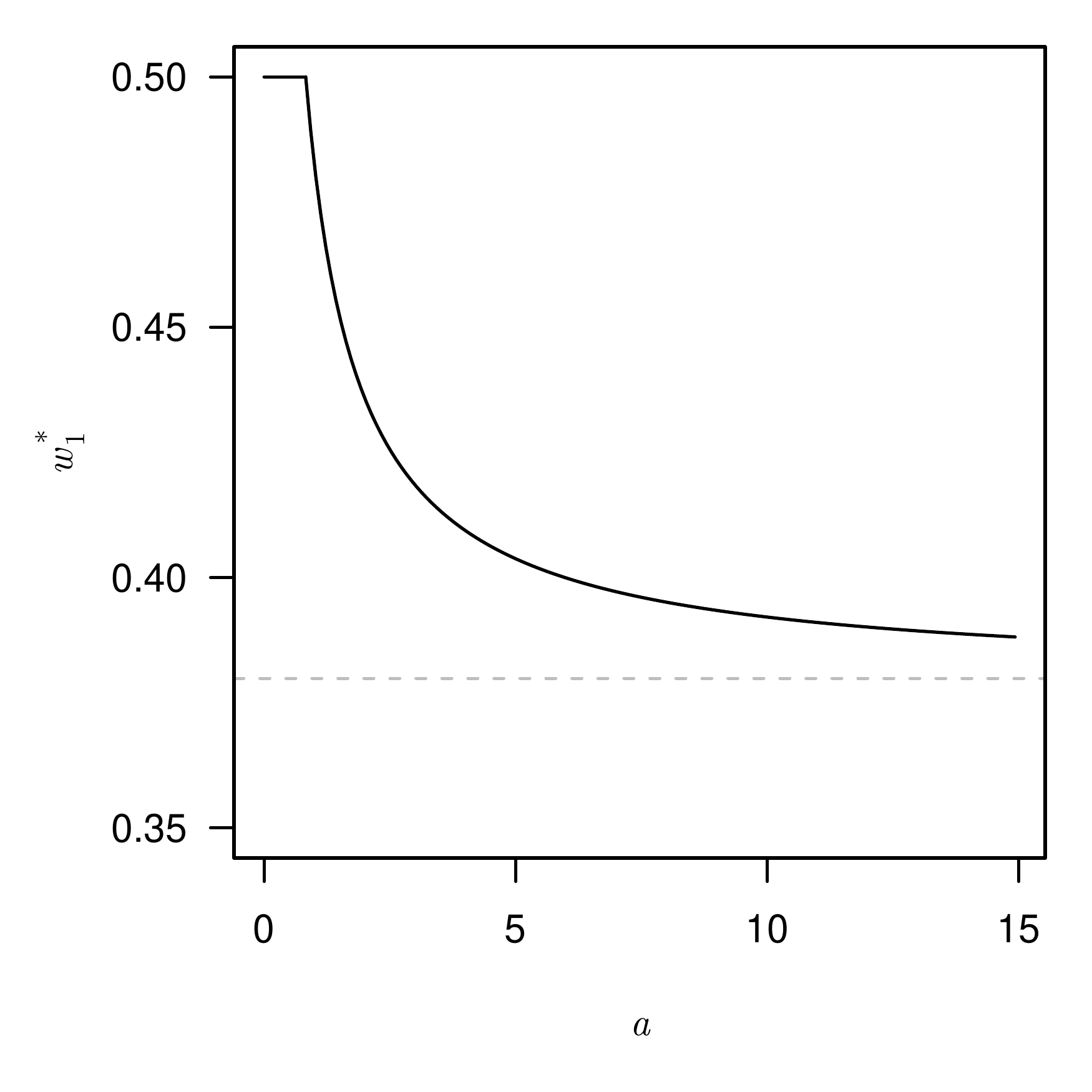}
  \caption{Optimal weight \(w_1^*=I_1^*/I\) with respect to
    \(a = I\tau^2\) for the straight line model with uncorrelated
    random effects, \(\rho = 0\), and \(J = 3\).  The dashed
    horizontal line indicates the limit
    \((10 - \sqrt{10})/18 \approx 0.3799\) for \(a\to\infty\).}
  \label{fig:xy-linreg}
\end{figure}
\begin{figure}
  \centering
  \includegraphics[width=0.45\textwidth]{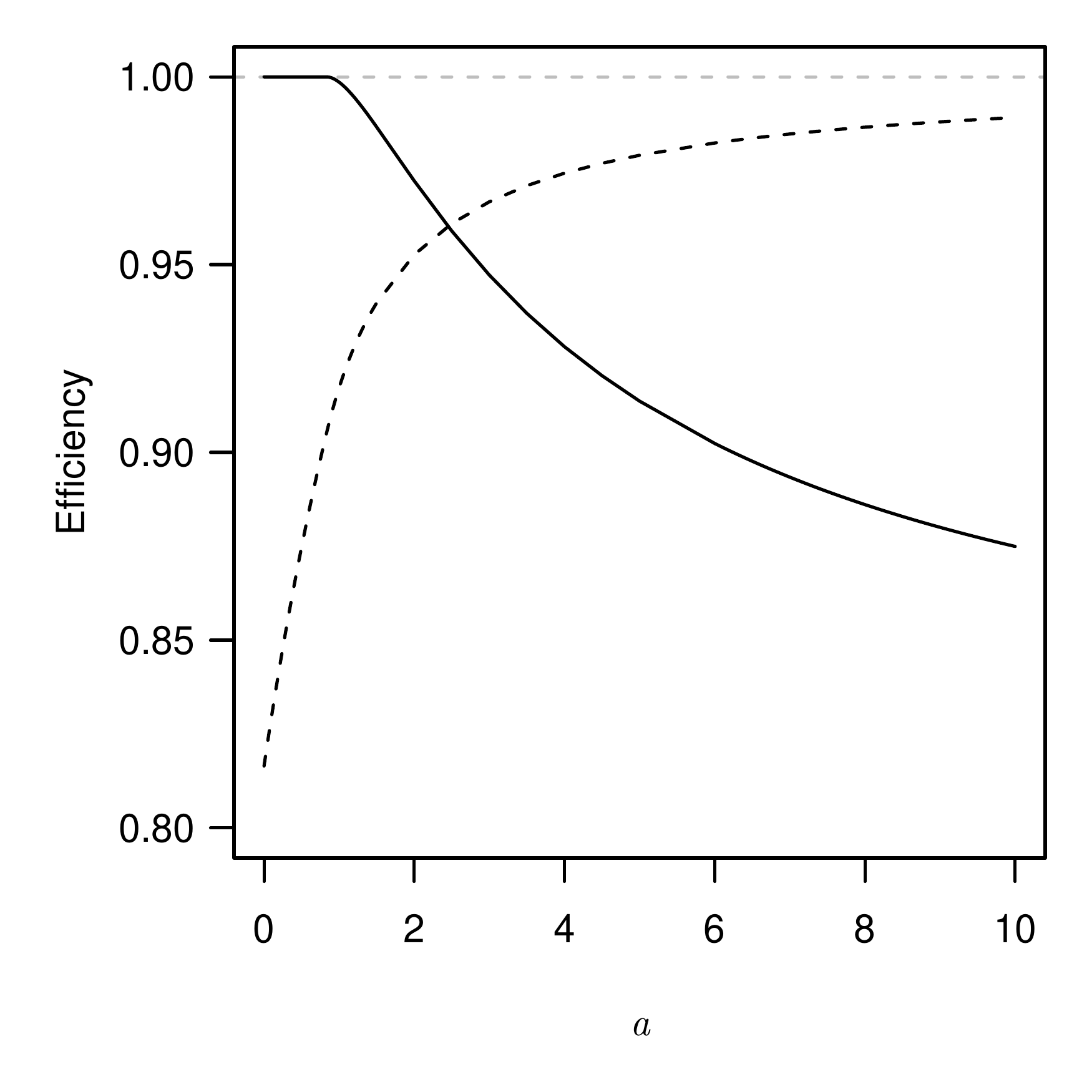}
  \caption{Efficiency of the minimally supported design (\(w_1=1/2\);
    solid) and the uniform design (\(w_1=1/3\); dashed) with respect
    to \(a = I\tau^2\) in the straight line model with uncorrelated
    random effects, \(\rho=0\), and \(J = 3\).}
  \label{fig:j3-null-eff}
\end{figure}

\subsection{Equivalence Theorem}
\label{sec:sub:EquivalenceTheorem}
For fully supported optimal designs, we can derive an equivalence
theorem which characterizes optimality in terms of an analytical
condition.  This result is motivated by the statement in
formula~(5.2.7) of \cite{FedorovHackl:1997}, p.~78, which covers the
situation \(J = p\) as a special case.

\begin{theorem}
  \label{theorem:equivalence}
  Let \(\bm{\beta}\) be estimable under the design
  \(\xi^* = \begin{pmatrix} I_{1}^* & \ldots &
    I_{J}^* \end{pmatrix}\).  Denote by \(\psi_j(\xi)\) the \(j\)th
  diagonal entry of the \(J \times J\) matrix
  \begin{equation}
    \label{eq:def-psi}
    \left(\sigma_{\varepsilon}^2 \mathbf{I}_{J} + \mathbf{\Sigma}_{\gamma} \mathbf{M}_{0}(\xi)\right)^{-1} \mathbf{A}_{\bm{\beta}}\, \mathbf{M}_{\bm{\beta}}(\xi)^{-1} \mathbf{A}_{\bm{\beta}}^{\mathrm{T}} \left(\sigma_{\varepsilon}^2 \mathbf{I}_{J} + \mathbf{M}_{0}(\xi) \mathbf{\Sigma}_{\gamma}\right)^{-1} \, ,
  \end{equation}
  then \(\xi^*\) is locally \(D\)-optimal if and only if
  \begin{equation}
    \label{eq:equivalence}
    \psi_j(\xi^*) \leq {\frac{1}{I}} \sum_{\ell = 1}^{J} I_{\ell}^* \psi_{\ell}(\xi^*)
  \end{equation}
  for all \(j=1, \ldots, J\).
  \par
  Moreover, for the optimal design \(\xi^*\), equality holds in
  \eqref{eq:equivalence} for those \(j\) for which \(I_j^* > 0\).
\end{theorem}

The proof is given in the Appendix.  For the general case of
potentially not fully supported designs, the essential arguments for
proving the equivalence are based on a suggestion by Norbert
\cite{Gaffke:2022}.

The equivalence theorem can be used to check whether numerically
obtained designs are indeed optimal as done in the next section.  For
fully supported designs, the inequalities can be replaced by equations
in~\eqref{eq:equivalence}.

\begin{corollary}
  \label{corollary:equivalence-fullsupport}
  Let \(\xi^*\) be fully supported, i.e.\ \(I_j^*>0\) for all \(j\),
  and let \(\mathbf{A}_{\bm{\beta}}\) have full column rank \(p\).
  Let \(\psi_j\) be defined as in Theorem~\ref{theorem:equivalence}.
  Then \(\xi^*\) is locally \(D\)-optimal if and only if all
  \(I \psi_j\) are equal,
  \(j=1, \ldots, J\).
\end{corollary}

In the case \(J = p\) of a minimal number of time points, we obtain a
characterization of the optimal design which is in accordance with
formula~(5.2.7) of \cite{FedorovHackl:1997}, p.~78, when there the
regression functions are substituted by the \(j\)th unit vector and
the factor \(I\) is introduced to account for the size of the testing.

\begin{corollary}
  \label{corollary:equivalence}
  Let \(J = p\).  Denote by \(\psi_j\) the \(j\)th diagonal entry of
  \begin{equation*}
    \mathbf{M}_{0}(\xi^*)^{-1}
    \left(\sigma_{\varepsilon}^2 \mathbf{M}_{0}(\xi^*)^{-1} + \mathbf{\Sigma}_{\gamma} \right)^{-1}  	\mathbf{M}_{0}(\xi^*)^{-1}
  \end{equation*}
  then \(\xi^*\) is \(D\)-optimal if and only if
  \begin{equation*}	
    I \psi_j = 
    \mathrm{trace}\left(
      \left(\sigma_{\varepsilon}^2 \mathbf{M}_{0}(\xi^*)^{-1} + \mathbf{\Sigma}_{\gamma} \right)^{-1}
      \mathbf{M}_{0}(\xi^*)^{-1} 	
    \right)
  \end{equation*}	
  for all \(j=1, \ldots, J\).
\end{corollary}

In particular, for \(J = p\), the \(D\)-optimal design \(\xi^*\) does
neither depend on the location parameters \(\bm{\beta}\) nor on the
particular model for the mean response curve.

\section{Computational Results}
\label{sec:computational}
In this section computational results are presented for both nonlinear
as well as for the straight line growth curve model considered before.
For the linear (straight line growth curve) model the \(D\)-optimal
design does not depend on \(\bm{\beta}\) while for the nonlinear
models we determine locally \(D\)-optimal designs for various values
of \(\bm{\beta}\).  The values used for the parameter are given in
Table~\ref{tab:paramSet}.  For the reason of expected increase in the
mean scores and with respect to standardization, only parameters with
\(\beta_1 > \beta_0 = 0\) were considered.  The correlation parameter
\(\rho\) and the variance ratio
\(\tau^2 = \sigma_\gamma^2 / \sigma_\varepsilon^2\) were varied to
exhibit their influence on the optimal design.
\begin{table}[tb]
  \centering
  \caption{Parameter settings used in the numerical computations}
  \begin{tabular}{l|ll|l}
    \(\beta_{0}\) & \(0\) & \(\rho\) &\(0, 0.05,\ldots, 0.95\)\\
    \(\beta_{1}\) & \(1, 3, 5, 10\)&\(\sigma_\gamma^2\) &\(0,0.1,\ldots,2,2.5,5,10\)\\
    \(\beta_{2}\) & \(0.5, 1, 2\) & \(\sigma_\varepsilon^2\)&\(1\)\\
    \(\beta_{3}\) & \(- 2, - 1, 0, 1, 2\) \\
  \end{tabular}
  \label{tab:paramSet}
\end{table}
All computations where done using R \citep{R:2020}.  The optimality of the designs obtained was checked
by applying Theorem~\ref{theorem:equivalence}.

Calculations were made for varying numbers \(J\) of time points and
number of items \(I\). Here we only present results for \(J=7\) time
points, which is motivated by the results of the meta analysis in
\cite{ScharfenJansenHolling:2018a}. Computations for other numbers
\(J > p\) of time points show similar results. For \(J = p\) the
theoretical results of Subsection~\ref{sec:sub:saturated} were
recovered. Also note, that the figures show results for $I=100$, which
seems to be a reasonable number of items for $J=7$ time points.  The
subsequent discussion of the results in the text is in terms of the
standardized variance ratio $a=I\tau^2$.  The values $\tau^2=0.1$,
$0.5$, $1$ and $2$ in the figures correspond to $a=10$, $50$, $100$
and $200$, respectively.

\begin{figure}
  \centering
  \includegraphics[width=0.45\textwidth]{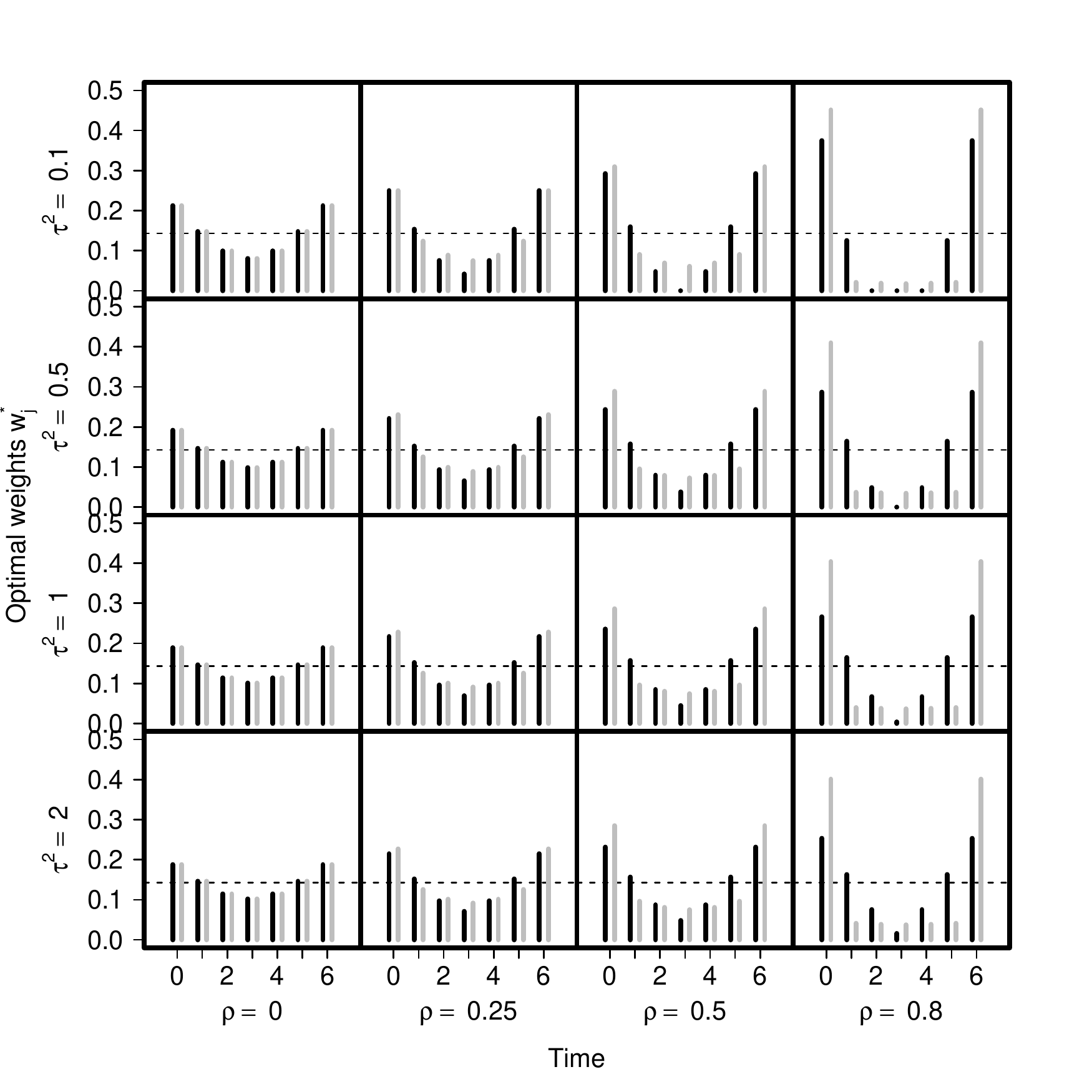}  
  \caption{Optimal weights \(w_j^* = I_j^*/I\) for the straight line
    model with \(J = 7\) and \(I = 100\) for different values of
    \(\tau^2\) and \(\rho\). Compound symmetry black; Autoregressive
    gray. The dashed horizontal line marks the weight \(1 / 7\)
    of the uniform design.}
  \label{fig:LineJ7}
\end{figure}

The general behavior of the optimal weights with respect to the
standardized variance ratio $a = I\tau^2$ and the correlation $\rho$
is already apparent in the straight line model (see
Figure~\ref{fig:LineJ7}).  If the standardized variance ratio $a$ is
small, the designs are closer to minimally supported designs which are
optimal in the case of no random effects and are uniform on \(p\) time
points for the present models.  In particular, some weights
\(w_j^* = I_{j}^* / I\) vanish for $a\rightarrow 0$.  When \(a\)
increases weights are spread out more equally across time points.  An
explanation for this behavior may be that variance of the random
effect is relatively large for large $a$.  Hence, the observations
vary more between time points than for small values of $a$. To get
precise estimates this is counteracted in the design by increasing the
number of observations spent at those time points. Note, that in the
straight line model the weights of the middle points decrease for
small $a$, and the minimally supported optimal design is concentrated
on the boundary points $t_1=0$ and $t_J=J-1$ only.

As \(\rho\) increases, observations at neighboring time points get
more similar and, again, optimal weights of some time points get
smaller. For the straight line model these are the interior time
points.  This qualitative behavior is the same for both covariance
structures under consideration.  For the autoregressive covariance
structure, the weights of time points closer to those of the minimally
supported design tend to be higher.  Overall the optimal weights under
compound symmetry tend to be closer to balanced designs. This is in
accordance with the observations for the unstructured model in
Subsection~\ref{sec:sub:Jlargerp}.

For the nonlinear models only asymmetric situations are displayed in
Figures~\ref{fig:4PLJ7} to \ref{fig:3PExpJ7_slopes} and more weight is
spent at the first time points. 

\begin{figure}
  \centering
  \includegraphics[width=0.45\textwidth]{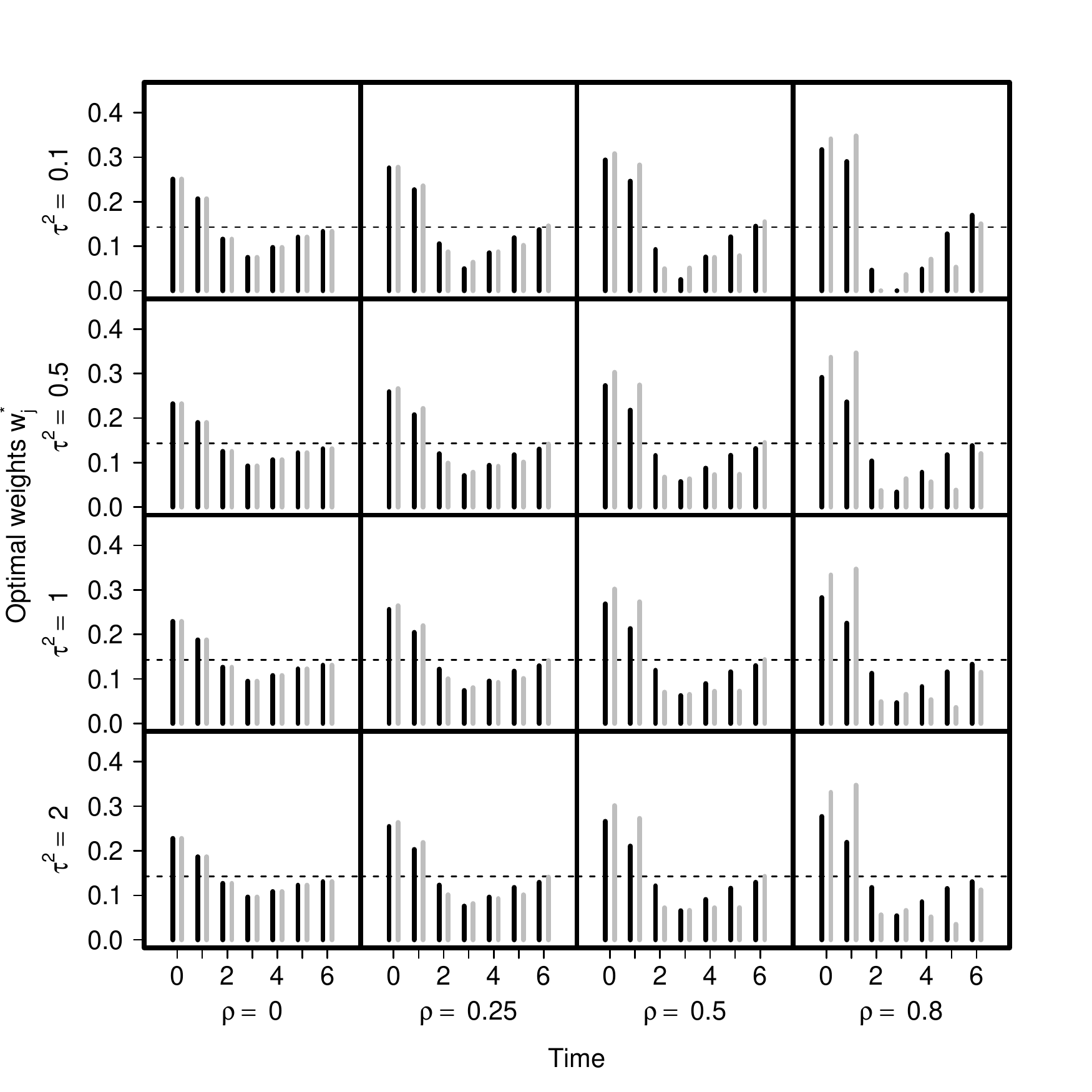}  
  \caption{Optimal weights \(w_j^* = I_j^*/I\) for the exponential
    model with \(\bm{\beta} = (0\;1\;1)^{\mathrm{T}} \), \(J=7\), and
    \(I = 100\) for different values of \(\tau^2\) and \(\rho\).
    Compound symmetry black; Autoregressive gray.  The dashed
    horizontal line marks the weight \(1 / 7\) of the uniform design.
  }
  \label{fig:3PExpJ7}
\end{figure}
\begin{figure}
  \centering
  \includegraphics[width=0.45\textwidth]{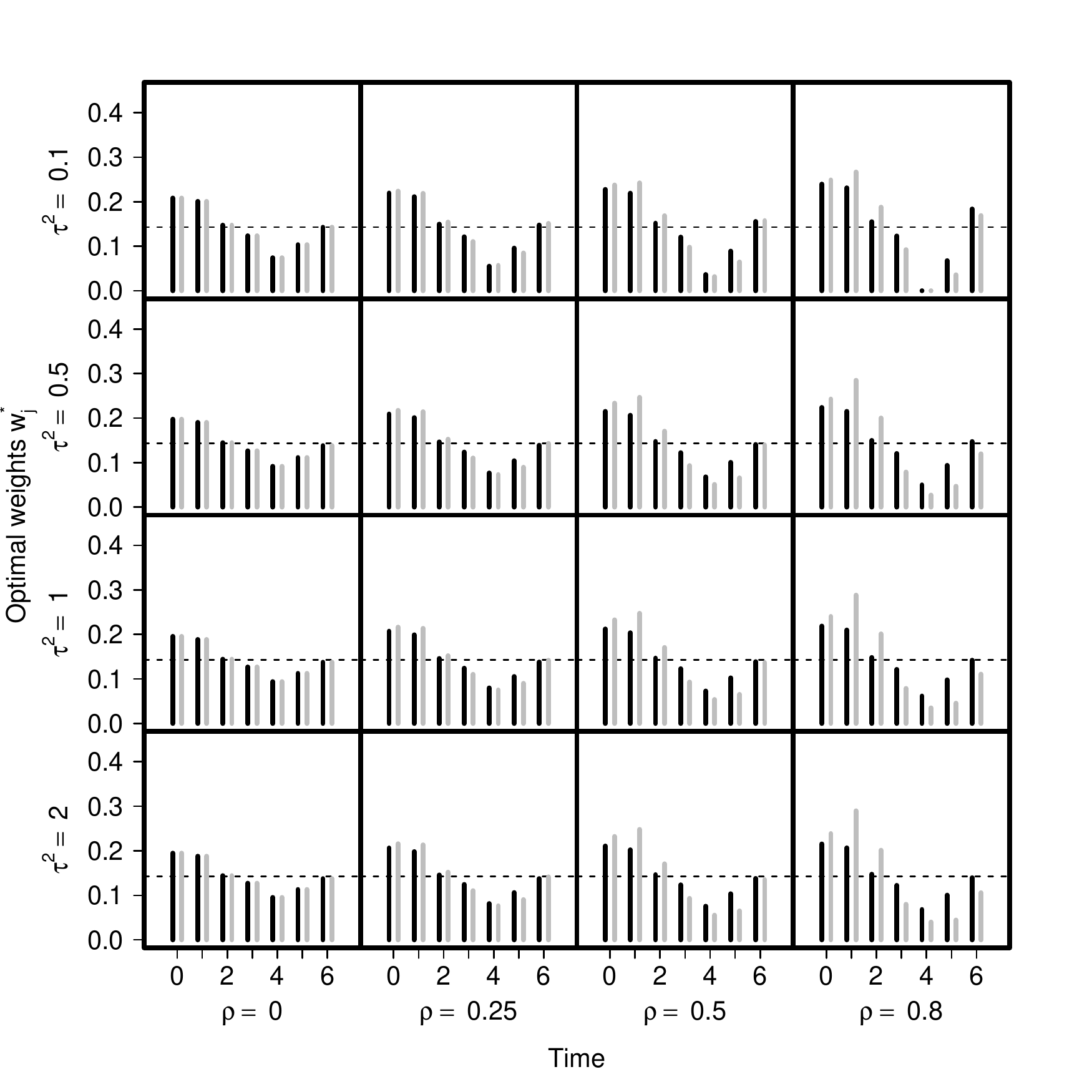}
  \caption{Optimal weights \(w_j^* = I_j^*/I\) for the logistic
    model with \(\bm{\beta} = (0\;1\;1\;0)^{\mathrm{T}} \),
    \(J=7\), and \(I = 100\) for different values of \(\tau^2\)
    and \(\rho\).  Compound symmetry black; Autoregressive gray.
    The dashed horizontal line marks the weight \(1 / 7\) of the
    uniform design.  }
  \label{fig:4PLJ7}
\end{figure}

For larger slope \(\beta_{2}\), the weight is concentrated at the
first few time points and at the last (Figures~\ref{fig:3PExpJ7}
and~\ref{fig:3PExpJ7_slopes} for the exponential model;
Figures~\ref{fig:4PLJ7} and~\ref{fig:4PLJ7_slopes} for the logistic
model).  If the slope \(\beta_{2}\) becomes smaller, more weight is
shifted to interior points. (Figure~\ref{fig:3PExpJ7_slopes} and
\ref{fig:4PLJ7_slopes}) This is sensible, because initially the mean
curve shows its steepest ascent, while at the last time point it is
close to the upper asymptote (see Figure~\ref{fig:mean} for the mean
curves. Solid lines correspond to the examples in Figures~
\ref{fig:3PExpJ7} and \ref{fig:4PLJ7}).  For other model parameters
than \(\beta_{2}\) the influence on the weights was found to be not
substantial in the numerical computations and is, hence, not shown
here.

\begin{figure}
  \centering
  \includegraphics[width=0.45\textwidth]{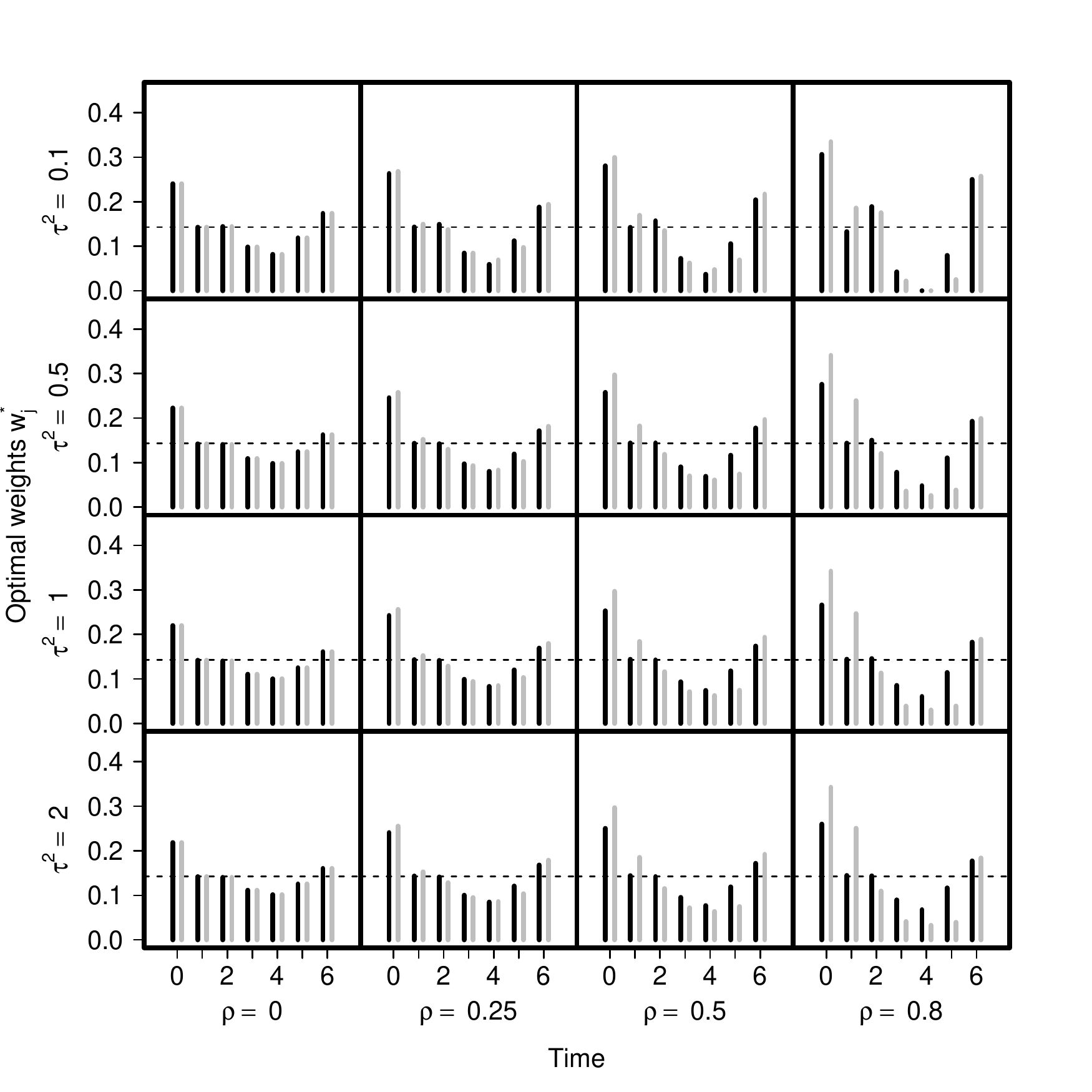}
  \includegraphics[width=0.45\textwidth]{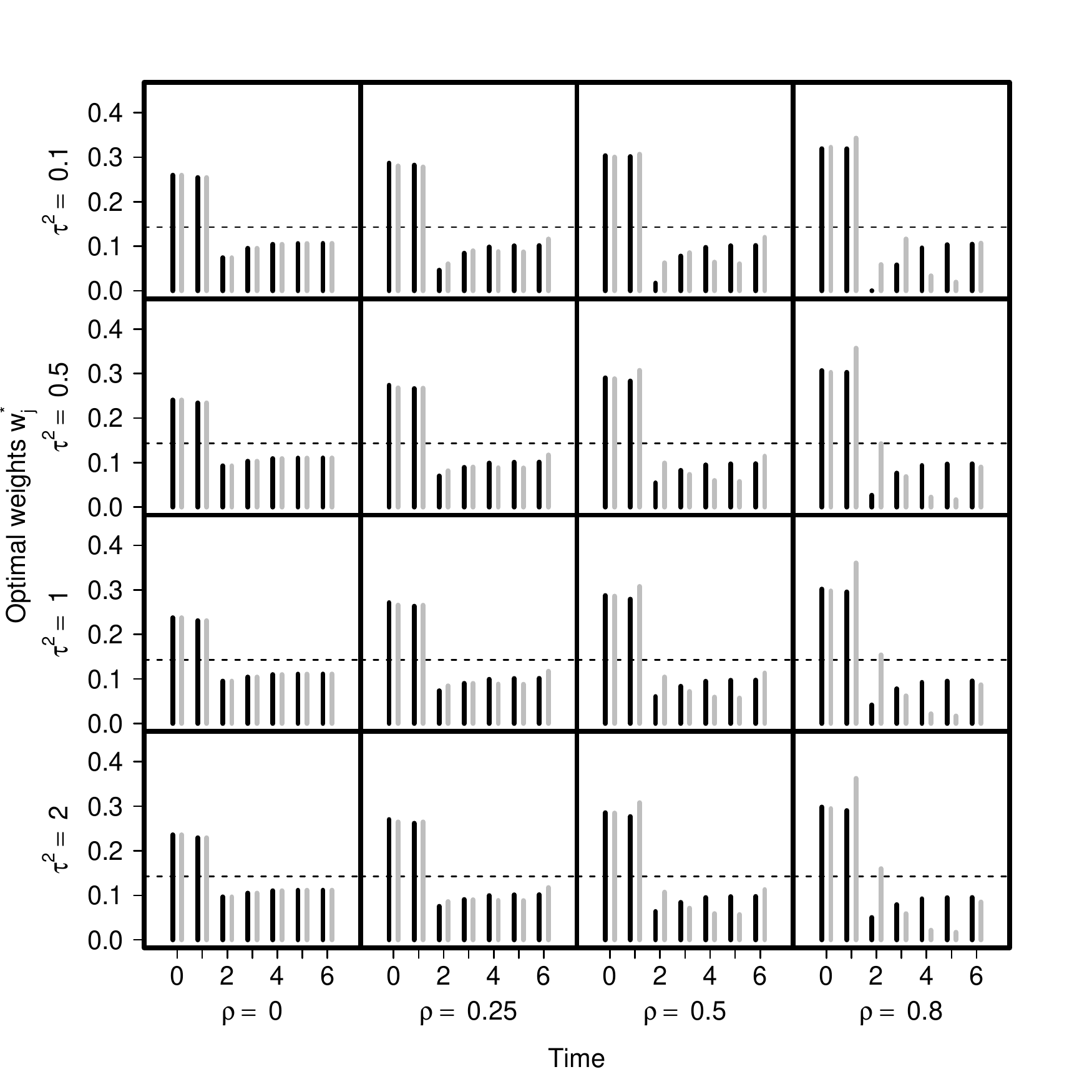}
  \caption{Optimal weights \(w_j^* = I_j^*/I\) for the exponential
    model with \(\bm{\beta} = (0\;1\;0.5)^{\mathrm{T}} \) (left) and
    \(\bm{\beta} = (0\;1\;2)^{\mathrm{T}} \) (right), \(J = 7\), and
    \(I = 100\) for different values of \(\tau\) and \(\rho\).
    Compound symmetry black; Autoregressive gray.  The dashed
    horizontal line marks the weight \(1 / 7\) of the uniform design.
  }
  \label{fig:3PExpJ7_slopes}
\end{figure}
\begin{figure}
  \centering
  \includegraphics[width=0.45\textwidth]{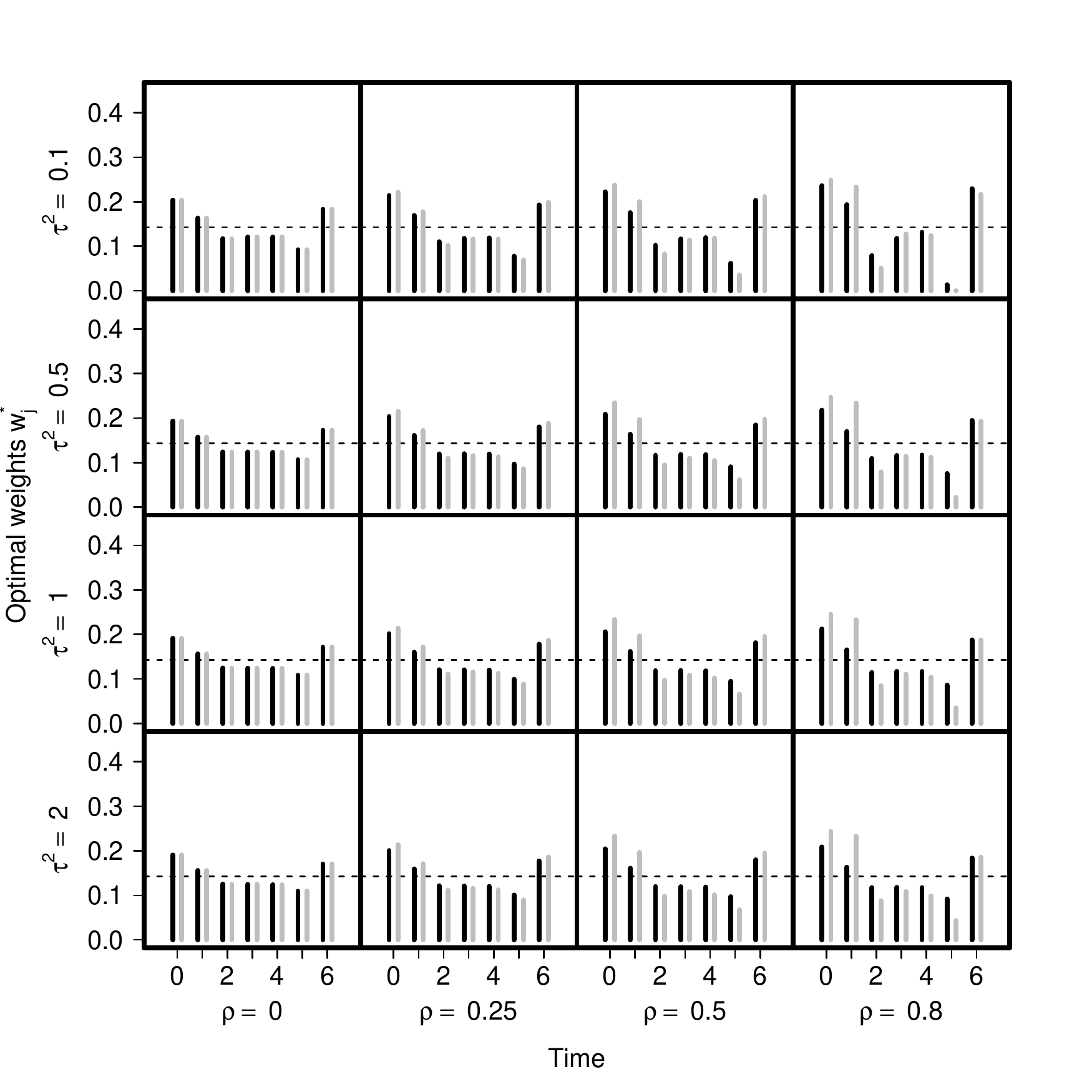}
  \includegraphics[width=0.45\textwidth]{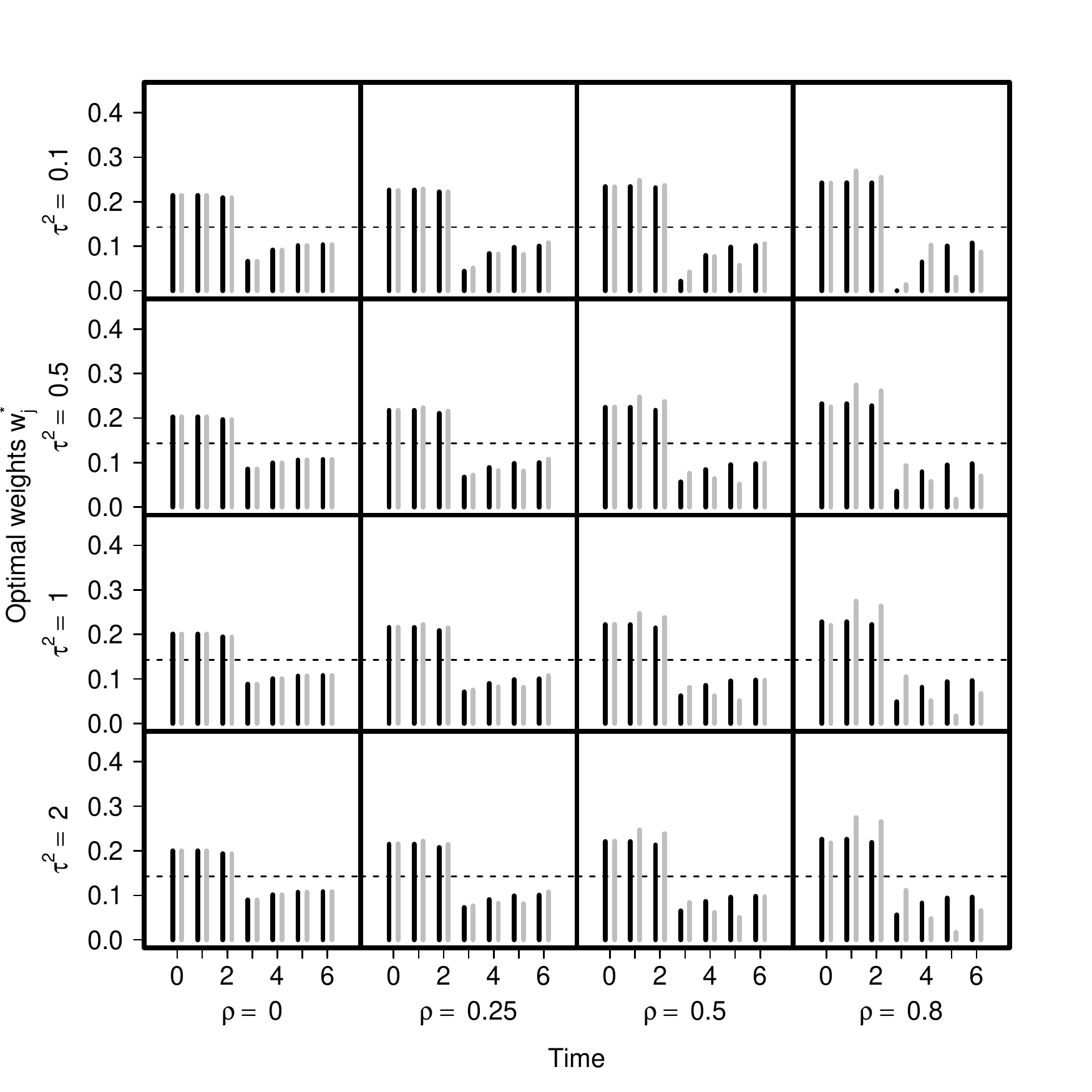}
  \caption{Optimal weights \(w_j^* = I_j^*/I\) for the logistic model
    with \(\bm{\beta}= (0\;1\;0.5\;0)^{\mathrm{T}} \) (left) and
    \(\bm{\beta}= (0\;1\;2\;0)^{\mathrm{T}} \) (right), \(J = 7\), and
    \(I = 100\) for different values of \(\tau^2\) and \(\rho\).
    Compound symmetry black; Autoregressive gray.  The dashed
    horizontal line marks the weight \(1 / 7\) of the uniform design.
  }
  \label{fig:4PLJ7_slopes}
\end{figure}

\begin{figure}
  \centering
  \includegraphics[width=0.45\textwidth]{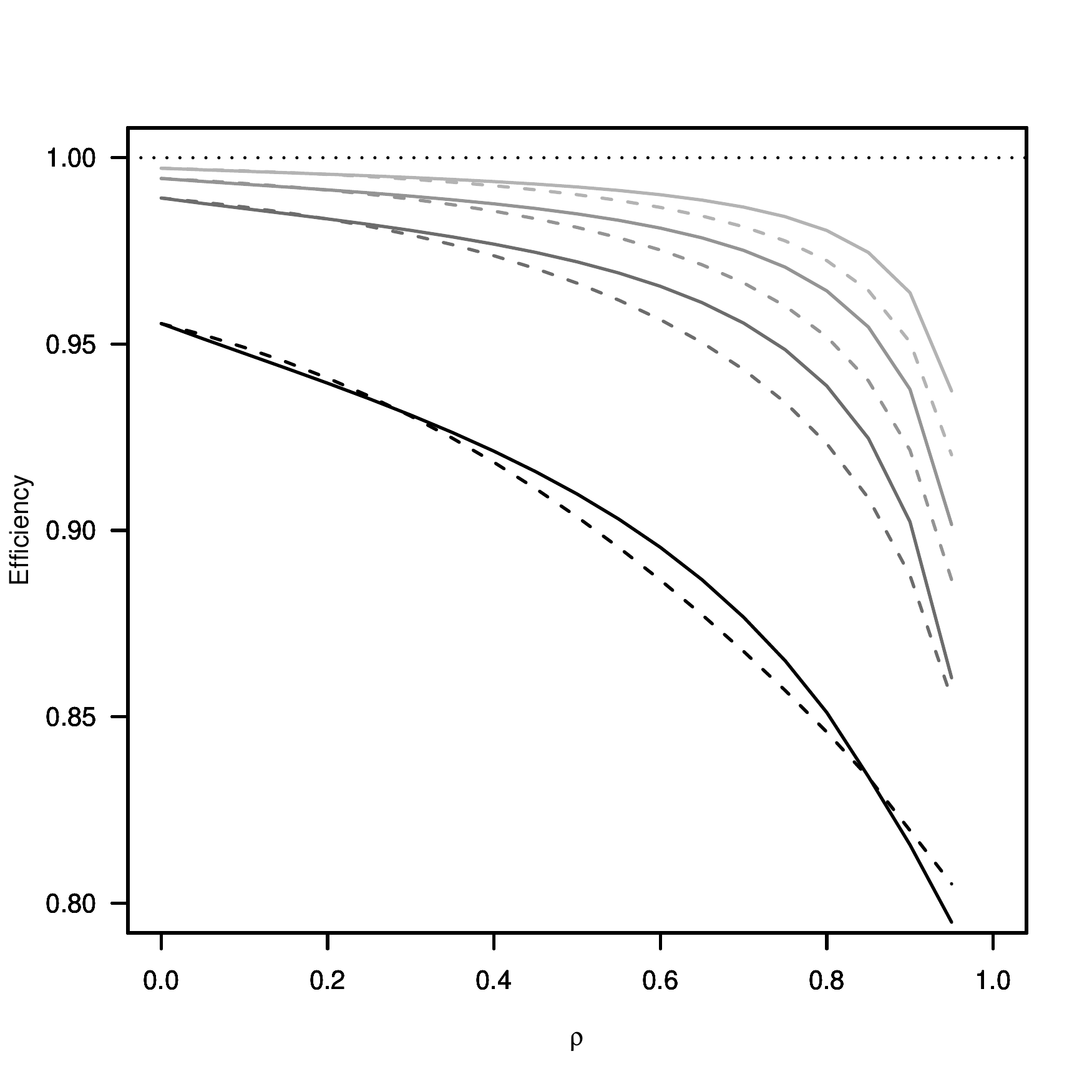}
  \includegraphics[width=0.45\textwidth]{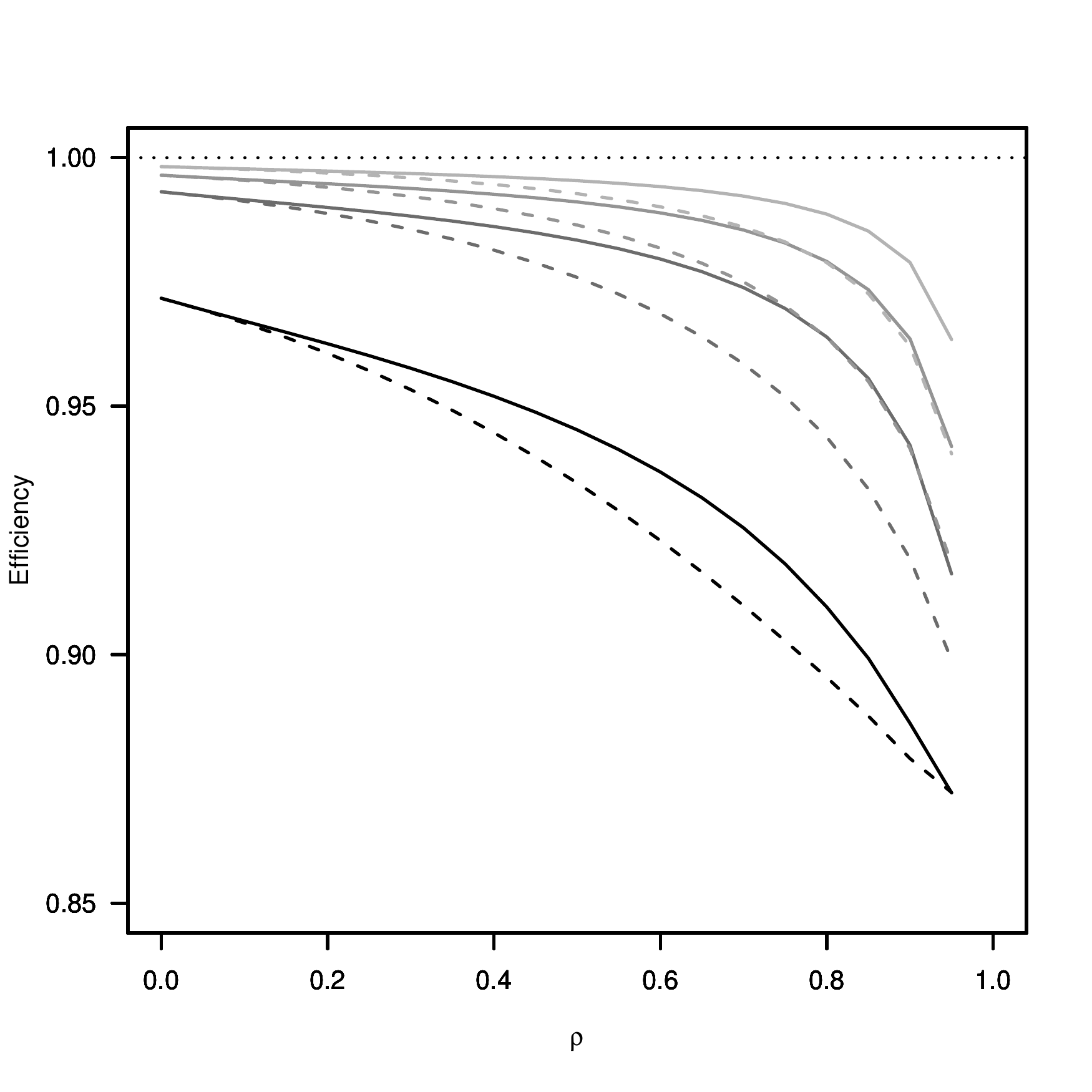}
  \caption{Efficiency of the uniform compared to the locally optimal
    design for \(J=7\) and \(I=100\) with respect to \(\rho\) for
    different values of \(\tau^2\). Gray shades change from
    \(\tau^2=0.1\) in black, over \(\tau^2=0.5\) and \(\tau^2=1\), to
    \(\tau^2=2\) in light gray. Solid lines correspond to compound
    symmetry, dashed lines to autoregressive  (left: exponential
    \(\bm{\beta}=(0\;1\;1)^{\mathrm{T}}\); right: logistic
    \(\bm{\beta}=(0\;1\;1\;0)^{\mathrm{T}}\)).}
  \label{fig:efficienciesJ7}
\end{figure}

Figure~\ref{fig:efficienciesJ7} displays the efficiency of the uniform
design in dependence on the correlation \(\rho\) for different values
of the parameters in both the exponential decay model (left panel) and
the sigmoid logistic model (right panel).  It can be seen from the
figures that the efficiency of the uniform design gets smaller for
larger values of the correlation \(\rho\) and gets larger for larger
values of the standardized variance ratio \(a = I \tau^2\).  The
latter behavior is in accordance with the observation that for larger
values of \(a = I \tau^2\) the observations are more spread over all
time points.  It is also in agreement with the analytical results
exhibited in Figure~\ref{fig:j3-null-eff} in the case of a straight
line growth curve for \(J=3\) and \(\rho=0\).  Moreover, the
efficiency of the uniform design is larger for compound symmetry than
for the autoregressive case.  Over all settings of the computations,
the efficiency of the uniform design turns out to be at least \(0.72\)
in the exponential decay and \(0.79\) in the sigmoid logistic model.

\section{Discussion}
\label{sec:discussion}
In this article, we have derived optimal designs for different 
latent growth models. For linear growth curves, the
optimal design only depends on the variances and correlations
of the random effects which may be appropriately assumed 
to be sufficiently known in advance for some studies.
In particular, the impact of the variance of the random effect 
on the optimal design is only through the standardized variance 
ratio \(a=I\sigma_{\gamma}^2/\sigma_{\varepsilon}^2\).

When the growth curve is nonlinear in the ability parameters,
the determination of optimal designs further requires prior
knowledge of the abilities. 
For most studies, this assumption may be considered as too 
restrictive. 
However, adjusting nominal values for
the ability parameters leads to locally optimal designs which
may serve as a benchmark and as a starting point
for the development of designs that require less prior information 
on the parameters to be estimated.
For this aim, three major approaches may come into question:
sequential, Bayesian and maximin efficient designs.
In the context of retesting, the use of sequential designs is 
apparently difficult because of repeated 
measurements.  However, incorporating prior
knowledge on the abilities within a Bayesian approach may be a 
promising strategy. There it seems to be quite reasonable
to impose a normal distribution
for the ability parameter.
Furthermore, maximin efficient designs are an alternative 
worth-wile considering
where the minimal efficiency is maximized
over a range of abilities. 
In the case of a Rasch model which also includes both an ability 
and a difficulty parameter,
\cite{GrasshoffHollingSchwabe:2012} have successfully
developed such maximin efficient as well as Bayesian
optimal designs.
Their approach may be adopted to the present situation.

In the case of a minimal number of time points (\(J = p\)), 
the present findings are in accordance with the results for
random coefficient regression models, 
see~\cite{EntholznerBendaSchmelterSchwabe:2005},
because the analysis of variance type regression functions 
of the random effects lie in the space spanned by the
regression functions of the fixed effects (mean response curve).
This is no longer the case if the number \(J\) of time points
exceeds the number \(p\) of fixed effects parameters.
Then arguments based on the shape of the fixed effects
regression functions do not longer apply and are 
superimposed by the properties of the analysis of variance type
random effects regression function which tend to spread 
observations more uniformly in optimal designs.
It would be challenging to derive general results
for such situations in which the random effects are not directly
associated with the fixed effects coefficients.
\vspace{5mm}

\noindent
\textbf{Acknowledgment:} This work was partly supported under grants
HO 1286/6 and SCHW 531/15 of the German Research Foundation (DFG).

\section*{Appendix: Proofs}
\label{sec:appendix}

\subsection*{Proof of Lemma~\ref{lem:infogeneralsupport}}
\label{subsec:appendix-a}
First we note that, in the case of fully supported designs
(\(I_j > 0\) for all \(j = 1, \ldots, J\)), \(\mathbf{M}_0\) is of
full rank, and the expression on the right hand side
of~\eqref{eq:infogeneralsupport} can be obtained by pre- and
post-multiplication of~\eqref{eq:covoftheta2} by
\(\mathbf{M}_0^{1/2}\mathbf{M}_0^{-1/2}\).

In the general case, we make use of the formula
\(\left(\mathbf{I} + \mathbf{C} \mathbf{C}^{\mathrm{T}}\right)^{-1} = \mathbf{I} - \mathbf{C} \left(\mathbf{I} + \mathbf{C}^{\mathrm{T}} \mathbf{C}\right)^{-1} \mathbf{C}^{\mathrm{T}}\)
for the inverse of a sum of matrices.
Let \(\mathbf{H}\) any matrix such that \(\mathbf{H}^{\mathrm{T}} \mathbf{H} = \mathbf{M}_0\) and set \(\mathbf{C} = \sigma_{\varepsilon}^{-1} \mathbf{H} \mathbf{\Sigma}_{\gamma}^{1/2}\), where \(\mathbf{\Sigma}_{\gamma}^{1/2}\) denotes a symmetric square root of \(\mathbf{\Sigma}_{\gamma}\), i.e.\ \(\mathbf{\Sigma}_{\gamma} = \mathbf{\Sigma}_{\gamma}^{1/2} \mathbf{\Sigma}_{\gamma}^{1/2}\).
Then we obtain 
\begin{equation*}
	 \left(\sigma_{\varepsilon}^{2} \mathbf{I}_{m} + \mathbf{H} \mathbf{\Sigma}_{\gamma} \mathbf{H}^{\mathrm{T}}\right)^{-1}
	=
	{\textstyle{\frac{1}{\sigma_{\varepsilon}^{2}}}}
	\left(
		\mathbf{I}_{m} 
		- \mathbf{H} \mathbf{\Sigma}_{\gamma}^{1/2} 
		\left(
			\sigma_{\varepsilon}^2 \mathbf{I}_{J}
			+ \mathbf{\Sigma}_{\gamma}^{1/2} \mathbf{H}^{\mathrm{T}} \mathbf{H} \mathbf{\Sigma}_{\gamma}^{1/2}
		\right)^{-1} 
		\mathbf{\Sigma}_{\gamma}^{1/2} \mathbf{H}^{\mathrm{T}}
	\right) \, ,
\end{equation*}
where \(m\) is the number of rows in \(\mathbf{H}\).
From this, we get
\begin{equation*}
	\mathbf{H}^{\mathrm{T}}  \left(\sigma_{\varepsilon}^{2} \mathbf{I}_{m} + \mathbf{H} \mathbf{\Sigma}_{\gamma} \mathbf{H}^{\mathrm{T}}\right)^{-1} \mathbf{H}
	=
	{\textstyle{\frac{1}{\sigma_{\varepsilon}^{2}}}}
	\left(
		\mathbf{M}_{0} 
		- \mathbf{M}_{0} \mathbf{\Sigma}_{\gamma}^{1/2} 
		\left(
			\sigma_{\varepsilon}^2 \mathbf{I}_{J}
			+ \mathbf{\Sigma}_{\gamma}^{1/2} \mathbf{M}_{0} \mathbf{\Sigma}_{\gamma}^{1/2}
		\right)^{-1} 
		\mathbf{\Sigma}_{\gamma}^{1/2} \mathbf{M}_{0}
	\right) \, ,
\end{equation*}
irrespectively of the particular choice of \(\mathbf{H}\).  If we set
\(\mathbf{H} = \mathbf{F}\) or \(\mathbf{H} = \mathbf{M}_{0}^{1/2}\),
we get the left or right hand side in~\eqref{eq:infogeneralsupport},
respectively.

\subsection*{Proof of Theorem~\ref{theorem:equivalence}}
\label{subsec:appendix-b}
The \(D\)-criterion is continuous and concave by
Lemma~\ref{lem:concave}.  We first start with the situation that the
optimal design \(\xi^*\) is fully supported and, hence, the
information matrix \(\mathbf{M}_{0}(\xi^*)\) in the corresponding
one-way layout is non-singular.  For a fully supported design
\(\xi = \begin{pmatrix} I_1 & \ldots & I_J \end{pmatrix}\),
\(I_j > 0\) for all \(j=1,\ldots,J\), the directional derivative of
\(\log\det(\mathbf{M}_{\bm{\beta}}(\xi))\) at \(\xi\) in the direction
of another design
\(\eta = \begin{pmatrix} I_1^{\prime} & \ldots &
  I_J^{\prime} \end{pmatrix}\) is given by
\begin{align*}
  & \mathrm{trace}\left(\mathbf{M}_{\bm{\beta}}(\xi)^{-1} \mathbf{A}_{\bm{\beta}}^{\mathrm{T}} \left(\sigma_{\varepsilon}^2 \mathbf{M}_{0}(\xi)^{-1} + \mathbf{\Sigma}_{\gamma}\right)^{-1} \sigma_{\varepsilon}^2 \mathbf{M}_{0}(\xi)^{-1} \left(\mathbf{M}_{0}(\eta) - \mathbf{M}_{0}(\xi)\right) \mathbf{M}_{0}(\xi)^{-1}
    \right.
  \\
  & \left. \hspace*{40mm} \mbox{} \times
    \left(\sigma_{\varepsilon}^2 \mathbf{M}_{0}(\xi)^{-1} + \mathbf{\Sigma}_{\gamma}\right)^{-1} \mathbf{A}_{\bm{\beta}}\right)
  \\
  & \hspace*{20mm} = \sum_{j = 1}^{J} I_j^{\prime} \psi_j(\xi) - \sum_{j = 1}^{J} I_j \psi_j(\xi) \, .
\end{align*}
As \(\mathbf{M}_{\eta,0}\) and, hence, the directional derivative is
linear in \(\eta\), the criterion is differentiable.  Thus, by
standard arguments of convex optimization, it is sufficient to
consider the directional derivative for single time point designs
\(\eta = \xi_j\) assigning all \(I\) observations to a single time
point \(t_j\).  Then the design \(\xi^*\) is \(D\)-optimal if (and
only if) the directional derivative at \(\xi = \xi^*\) is less or
equal to zero for all \(\eta = \xi_j\).  This is equivalent to the
condition that \(I \psi_j(\xi^*)\) is less or equal to
\(\frac{1}{I} \sum_{j = 1}^{J} I_j \psi_j(\xi^*)\)
for all \(j = 1, \ldots, J\).

For the general situation, in which the optimal design \(\xi^*\) may
be supported by subset of the time points, we follow an idea proposed
by \cite{Gaffke:2022} based on the use of subgradients which we sketch
here.  By Lemma~\ref{lem:concave} the \(D\)-criterion is continuous
and concave on the set \(\Xi_{\bm{\beta}}\) of all designs for which
\(\bm{\beta}\) is estimable.  Hence,
\(-\log\det(\mathbf{M}_{\bm{\beta}}(\xi))\) is a ``closed proper
convex function'' in the spirit of \cite{Rockafellar:1970}.  The
vector
\(\bm{\psi}(\xi) = \begin{pmatrix} \psi_1(\xi), \ldots,
  \psi_j(\xi) \end{pmatrix}^{\mathrm{T}}\) is the gradient at \(\xi\)
of the \(D\)-criterion on the set of fully supported designs and can
be continuously extended to \(\Xi_{\bm{\beta}}\) by its
definition~\eqref{eq:def-psi}.
	
By Theorem~25.6 in \cite{Rockafellar:1970}, at not fully supported
designs \(\xi \in \Xi_{\bm{\beta}}\), the set of subgradients is given
by
\(\{\bm{\psi}(\xi) + \mathbf{v}; \mathbf{v} \in \mathcal{V}(\xi)\}\),
where \(\mathcal{V}(\xi)\) is the set of all
\(\mathbf{v} = \begin{pmatrix} v_1, \ldots,
  v_J \end{pmatrix}^{\mathrm{T}}\) such that \(v_j = 0\) for
\(I_j > 0\) and \(v_j \geq 0\) for \(I_j = 0\).  Then, according to
Theorem~27.4 in \cite{Rockafellar:1970}, the design \(\xi^*\) is
\(D\)-optimal if and only if there is a vector
\(\mathbf{v} \in \mathcal{V}(\xi)\) such that
\(\sum_{j = 1}^{J} I_j (\psi_j(\xi^*) + v_j) \leq \sum_{j = 1}^{J}
I_j^* (\psi_j(\xi^*) + v_j)\) for all designs
\(\xi = \begin{pmatrix} I_1 & \ldots & I_J \end{pmatrix}\) Now,
\(\sum_{j = 1}^{J} I_j v_j \geq 0\) and
\(\sum_{j = 1}^{J} I_j^* v_j = 0\) for any
\(\mathbf{v} \in \mathcal{V}(\xi)\).  Hence, \(D\)-optimality of
\(\xi^*\) is equivalent to
\(\sum_{j = 1}^{J} I_j \psi_j(\xi^*) \leq \sum_{j = 1}^{J} I_j^*
\psi_j(\xi^*)\) for all \(\xi\).  In the equivalence condition we may
again restrict to single time point designs \(\xi = \xi_j\) which
establishes~\eqref{eq:equivalence}.
	
Furthermore, for an optimal design \(\xi^*\) equality follows
in~\eqref{eq:equivalence} for all support points of \(\xi^*\).

\end{document}